\documentclass[12pt,oneside]{amsart}

\usepackage{cite}

\usepackage[charter]{mathdesign}
\usepackage{tikz}
\usepackage{pgfplots}
\usetikzlibrary{positioning, calc, cd}
\usepackage[all,cmtip]{xy}
\usepackage{graphicx}

\title{What is connectivity?}

\author{Jean F. Du Plessis}
\address{Center for Theoretical Physics, Massachusetts Institute of Technology, Cambridge, MA 02139, USA}
\email{jeandp@mit.edu}

\author{Zurab Janelidze}
\address{Department of Mathematical Sciences\\ Stellenbosch University, South Africa 
and  
National Institute for Theoretical and Computational Sciences (NITheCS), Stellenbosch, South Africa}
\email{zurab@sun.ac.za}

\author{Bernardus A. Wessels}
\address{Department of Pure Mathematics and Mathematical Statistics, Centre for Mathematical Sciences, Cambridge, UK and Department of Mathematical Sciences, Stellenbosch University, South Africa}
\email{baw47@cam.ac.uk}

\makeatletter
\@namedef{subjclassname@2020}{
  \textup{2020} Mathematics Subject Classification}
\makeatother

\subjclass[2020]{54D05, 06A06, 06B23, 18A40, 06D22}
\keywords{}

\usepackage{amssymb, thmtools, color, enumitem, hyperref, graphicx}
\usepackage{tasks}
\usepackage{float}
\usepackage{amsfonts}
\usepackage[all,cmtip]{xy}
\hypersetup{pdfborder=0 0 0}
\usepackage{booktabs} 
\usepackage{enumitem} 

\usepackage{geometry}
\setlist{leftmargin=1cm}

\declaretheoremstyle[
  headformat = \textcolor{red}{\NUMBER. }\NAME{}\NOTE,
  headindent=0.8cm
]{actstyle}
\theoremstyle{actstyle}
\newtheorem{theorem}{Theorem}

\newtheorem{corollary}[theorem]{Corollary}

\newtheorem{example}[theorem]{Example}

\newtheorem{definition}[theorem]{Definition}

\newcommand{\powerset}[1]{\operatorname{\mathcal{P}}(#1)}

\begin{document}
\vspace{-\baselineskip}
\vspace{-\baselineskip}
\vspace{-\baselineskip}

\null\hfill\begin{tabular}[t]{l@{}}
  \text{MIT-CTP/5835}
\end{tabular}

\vspace{\baselineskip}
\vspace{\baselineskip}

\maketitle

\vspace{-\baselineskip}

\begin{abstract}
In this paper, we explore a taxonomy of connectivity for space-like structures. It is inspired by isolating posets of connected pieces of a space and examining its embedding in the ambient space. The taxonomy includes in its scope all standard notions of connectivity in point-set and point-free contexts, such as connectivity in graphs and hypergraphs (as well as $k$-connectivity in graphs), connectivity and path-connectivity in topology, and connectivity of elements in a frame.
\end{abstract}

\makeatletter
  \def\l@subsection{\@tocline{2}{0pt}{4pc}{5pc}{}}
\makeatother

\section*{Introduction}

Point-free topology generalizes various constructions, properties, and results about topological spaces to frames. The generalization is given by replacing the frame of open subsets of a topological space, ordered under subset inclusion, with an abstract frame --- a complete lattice where finite meets distribute over infinite joins. See \cite{picado_pultr_2012} for an overview of this theory. 

There is a line of work in the general study of connectivity, which takes a slightly different approach to generalizing connected sets than how point-free topology generalizes open sets. Instead of extracting the poset of connected objects and discarding the ambient structure, one abstracts the ambient structure to a complete lattice of its `pieces' (whether connected or not) and isolates a subset of connected elements in the lattice. In general, this subset is not determined uniquely by the lattice. Instead, one is interested in different kinds of such subsets, each giving rise to a notion of connectivity. See \cite{10.2307/1969257,article,Hammer1968Sep,borger_1983,matheron_1985,matheron_serra_1988,Serra1998Nov,dugowson2010connectivity,stadler_stadler_2015} for such an approach to a general theory of connectivity. In most of these works, the ambient complete lattice is the powerset of the underlying set of a mathematical structure in which one is interested to study connectivity. So the existing general theory of connectivity is largely a \emph{point-set} theory: connected objects are subsets of a given set whose elements represent `points' of a mathematical structure. A move to a \emph{point-free} theory is made by J.~Serra in \cite{Serra1998Nov}. In his approach, the lattice no longer need be a powerset lattice; this is motivated by applications to digital imaging, where general lattices allow inclusion of applications dealing with multicolored images. Almost ten years earlier, in \cite{borger1989}, shortly after his pioneering work on general point-set connectivity \cite{borger_1983}, R.~B\"orger proposes a general theory of connectivity in categories --- thus, a `higher' point-free theory. However, in \cite{borger1989}, there is essentially a single general theorem, and the special case in which the categories are posets is not analyzed (instead, other significant categorical examples are explored). 

In this paper we build on Serra's work on point-free connectivity and explore some foundational questions for a general theory whose aim is to study and compare different point-set and point-free notions of connectivity (without considering the broader, categorical situations). Our findings suggest that such a theory can be anchored at what we call a `connectivity lattice' --- potentially the most general notion of point-free connectivity. This context for connectivity turns out to be what you get when you suitably specialize B\"orger's categorical theory to posets. Extension of our work to categories is left for the future. 

A connectivity lattice is defined as a pair $(L,\mathcal{C})$, where $L$ is a complete lattice and $\mathcal{C}$ is a subset of $L$ (called a \emph{connectivity} in $L$), whose elements are called \emph{connected} elements of $L$. We impose only one axiom on this structure: every connected element below an element $x\in L$ is always below a unique maximal connected element below $x$. This condition essentially says that every element can be decomposed into `connected components'; such decomposition gives rise to a Galois connection
$$\mathcal{D}(\mathcal{C})\rightleftarrows L,$$
where $\mathcal{D}(\mathcal{C})$ denotes the poset of \emph{totally mail-disconnected} subsets $S$ of $\mathcal{C}$. By \emph{mail-connectedness} of a subset of $\mathcal{C}$ we mean connectedness relative to the graph structure on $\mathcal{C}$ where an edge between two elements of $\mathcal{C}$ represents that these elements have a common lower bound in $\mathcal{C}$ (see Figure~\ref{fig:mail-connectedness}). So `totally mail-disconnected' refers to the case when no two distinct elements have a common lower bound. 

We call these Galois connections \emph{connectivity adjunctions}. The left adjoint maps an element $S$ of $\mathcal{D}(\mathcal{C})$ to its join $\bigvee S$, while the right adjoint assigns to an element $x\in L$ the set $\mathcal{C}(x)$ of its \emph{connected components}, i.e., maximal connected elements below $x$. In fact, existence of such Galois connection turns out to be a characteristic property of connectivity lattices. Now, Serra's notion of connectivity from \cite{Serra1998Nov} can then be characterized by further requiring that every element of $L$ and the bottom element of $\mathcal{D}(\mathcal{C})$ are Galois closed for the connectivity adjunction.

The approach to connectivity through these Galois connections, as far as we are aware, is new to the field of point-free connectivity. It helps to uncover what seems to be an interesting conceptual taxonomy of connectivity (see Figure~\ref{figA}).
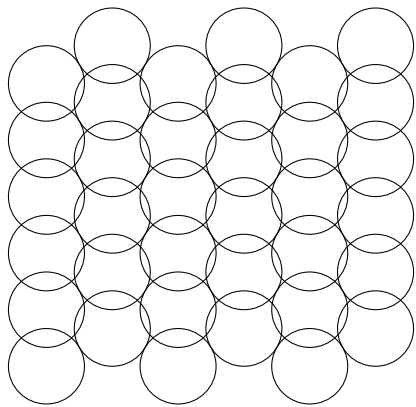
\begin{figure}
\[
\begin{tikzpicture}
    \def\n{5} 
    \def\radius{0.5} 
    \foreach \row in {0,...,\n}
        \foreach \col in {0,...,\n} {
            \pgfmathsetmacro\xshift{\col*2*\radius*cos(30)}
            \pgfmathsetmacro\yshift{\row*1.5*\radius}
            \ifodd\col
                \pgfmathsetmacro\yshift{\yshift+\radius}
            \fi
            \draw (\xshift, \yshift) circle (\radius);
        }
\end{tikzpicture}
\]
\caption{Venn diagram visualization of the notion of mail-connectedness in the poset of subsets of a given set. The set of regions marked by the circles is mail-connected, because any two regions can be linked by a sequence of overlapping regions.}
\label{fig:mail-connectedness}
\end{figure}  
One of these taxonomic groups is given by the case when the connectivity adjunction is an isomorphism. In this case, we call the corresponding connectivity lattice $(L,\mathcal{C})$ an \emph{absolute} connectivity lattice. A striking property of absolute connectivity lattices is that for each $L$ there can be at most one compatible $\mathcal{C}$. For a frame $L$ of open sets in a locally connected topological space, $\mathcal{C}$ turns out to be the set of open connected sets. More generally, a frame together with its set of connected elements forms an absolute connectivity lattice. This includes all complete atomic Boolean algebras, in which case $\mathcal{C}$ is the set of atoms. A more primitive class of absolute connectivity lattices is given by adjoining to a complete lattice a new bottom element. In this case, $\mathcal{C}$ is the set of old elements.

A culminating result in the paper is that isomorphism classes of absolute connectivity lattices are in one-to-one correspondence with isomorphism classes of posets having the following property:
\begin{enumerate}[label=(C)]
   \item \label{prop:C} Every mail-connected subset of the poset has a join.
\end{enumerate}
This property can be viewed as a form of `local completeness' of a poset. We can make this precise: as we prove in this paper, a poset has the property \ref{prop:C} if and only if for any element $x$ in the poset, $x^\uparrow$ is a complete lattice. Thus, a poset has the property \ref{prop:C} if and only if the dual poset, viewed as a category, is locally complete in the sense of \cite{freyd_scedrov_1990} (i.e., the subobject posets are complete lattices). The term `locally complete' is used in the latter sense for preorders in \cite{clementino_janelidze_2024}: a poset has the property \ref{prop:C} above if and only if its dual is a locally complete preorder in the sense of \cite{clementino_janelidze_2024}. We avoid the use of the term `locally complete' for two reasons: firstly, it is not suggestive of connectedness, and secondly, it has other meanings as well in order theory and topology. Instead, inspired by Figure~\ref{fig:mail-connectedness}, we call posets having the property \ref{prop:C} --- \emph{chainmails}.

A more appropriate categorical perspective on chainmails, which fits well with B\"orger's categorical theory \cite{borger1989}, seems to be that chainmails are the same as posets which, when viewed as categories, admit colimits of connected diagrams. Indeed, requiring existence of joins of mail-connected subsets of a poset is easily equivalent to requiring existence of joins of connected subsets.

The one-to-one correspondence between isomorphism classes of chainmails and those of absolute connectivity lattices is given by the mapping $(L,\mathcal{C})\mapsto\mathcal{C}$. In fact, the poset of connected elements forms a chainmail even for a general connectivity lattice! The mapping $(L,\mathcal{C})\mapsto\mathcal{C}$ is actually a functor from the category $\mathbf{Con}$ of connectivity lattices to the category $\mathbf{Chm}$ of chainmails. It has a left adjoint given by the mapping $\mathcal{C}\mapsto\mathcal{D}(\mathcal{C})$. The unit of this adjunction is an isomorphism, while the counit is given by the left Galois adjoint $\mathcal{D}(\mathcal{C})\to L$ in the connectivity adjunction. The one-to-one correspondence between isomorphism classes of chainmails and those of absolute connectivity lattices is then given by the equivalence of categories induced by this adjunction. We include these remarks only in a brief, at the end of the paper, since looking into categorical aspects of our work falls outside of the scope of the present work. 


Actually, chainmails turn out to be useful not only for classifying absolute connectivity, but also for providing a short answer to the question posed in the title: \emph{what is connectivity?} It turns out that a subset $\mathcal{C}$ of a complete lattice $L$ is a connectivity if and only if $\mathcal{C}$ is a subchainmail of $L$ (i.e., $\mathcal{C}$ is closed in $L$ under the types of joins required in \ref{prop:C}).


\section{Chainmails}\label{secA}  

Intuitively, in a space-like mathematical structure, a connected region is one that cannot be cleanly parted into disjoint regions. Suppose then a family of connected regions is such that any two of the regions are linked by a sequence of regions where any two consecutive regions in the sequence overlap. Any way of cleanly parting the union of these regions will cleanly part at least one of them (see Figure \ref{fig:parting}).
\begin{figure}
    \centering
    \includegraphics[width=0.8\textwidth]{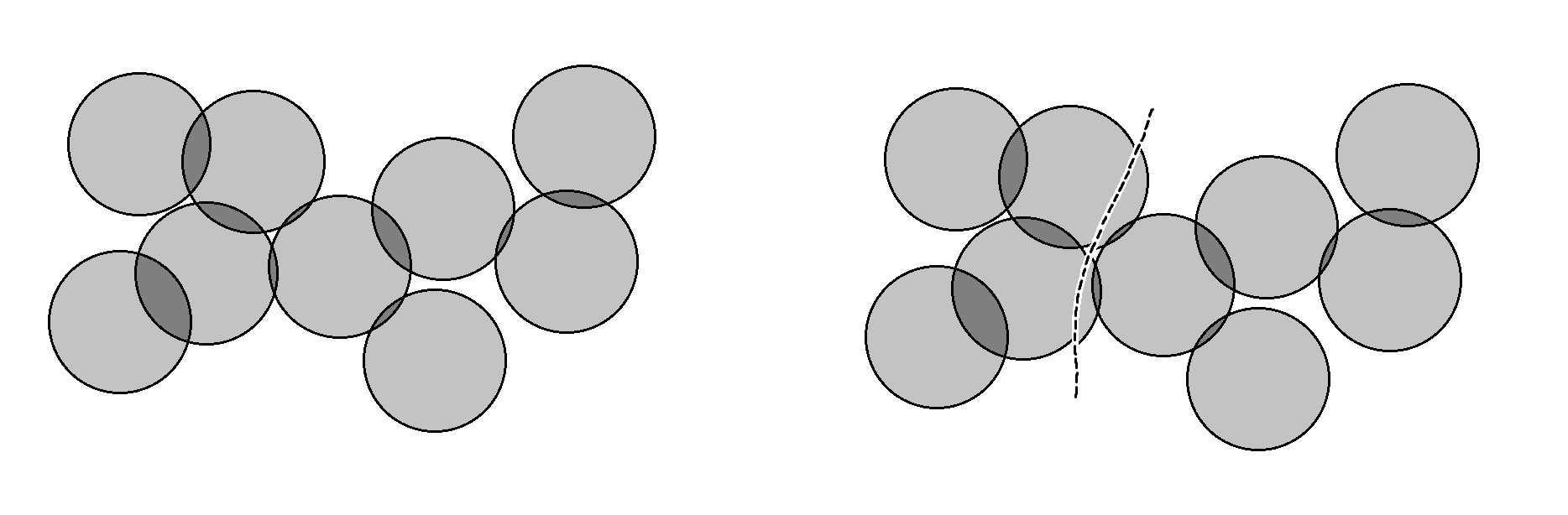}
    \caption{Illustration showing how certain unions of connected regions cannot be cleanly parted without cleanly parting at least one of them}
    \label{fig:parting}
\end{figure}
Therefore, the union cannot be cleanly parted into disjoint regions either and so it is also a connected region. This suggests that the poset of connected regions must have the property \ref{prop:C} from the Introduction.

\begin{definition}\label{def:Chainmail}
In a poset, a \emph{mail} is a non-empty set of elements of the poset having a common lower bound. A set $M$ of elements in the poset is \emph{mail-connected} when $M\neq\varnothing$ and for any two elements $x,y\in M$ there is a sequence $s_1,\dots,s_n\in M$ such that $x=s_1$, $y=s_n$, and $\{s_i,s_{i+1}\}$ is a mail for each $i\in\{1,\dots,n-1\}$. Such sequence $s_1,\dots,s_n$ is called a \emph{path} from $x$ to $y$. A \emph{chainmail} is a poset in which every mail-connected set has a join.
\end{definition}

It turns out that in the definition of a chainmail it is sufficient to require that all mails have joins.

\begin{theorem}\label{thmB} A poset is a chainmail if and only if every mail in it has a join.
\end{theorem}

\begin{proof} Suppose every mail has a join in the given poset. Then every path has a join. It can be obtained by hierarchically joining mails, as illustrated in the case of a path of length four:
\[
\scalebox{0.8}{\begin{tikzpicture}[row sep={2.1cm,between origins}, column sep={2.1cm,between origins}, every node/.style={draw, circle, inner sep=1pt, minimum size=20pt}]

\node (a1) {$a_0$};
\node[right=of a1] (a2) {$a_1$};
\node[right=of a2] (a3) {$a_2$};
\node[right=of a3] (a4) {$a_3$};
\node[right=of a4] (a5) {$a_4$};

\node[above=of $(a1)!0.5!(a2)$] (b1) {$b_0$};
\node[above=of $(a2)!0.5!(a3)$] (b2) {$b_1$};
\node[above=of $(a3)!0.5!(a4)$] (b3) {$b_2$};
\node[above=of $(a4)!0.5!(a5)$] (b4) {$b_3$};

\draw (a1) -- (b1) -- (a2);
\draw (a2) -- (b2) -- (a3);
\draw (a3) -- (b3) -- (a4);
\draw (a4) -- (b4) -- (a5);

\node[above=of $(b1)!0.5!(b2)$] (c1) {$c_0$};
\node[above=of $(b2)!0.5!(b3)$] (c2) {$c_1$};
\node[above=of $(b3)!0.5!(b4)$] (c3) {$c_2$};

\draw (b1) -- (c1) -- (b2);
\draw (b2) -- (c2) -- (b3);
\draw (b3) -- (c3) -- (b4);

\node[above=of $(c1)!0.5!(c2)$] (d1) {$d_0$};
\node[above=of $(c2)!0.5!(c3)$] (d2) {$d_1$};

\draw (c1) -- (d1) -- (c2);
\draw (c2) -- (d2) -- (c3);

\node[above=of $(d1)!0.5!(d2)$] (top) {$e$};

\draw (d1) -- (top);
\draw (d2) -- (top);

\end{tikzpicture}}
\]
Consider a mail-connected set $C$ and pick any element $c\in C$. For each $d\in C$ there is a path that connects $c$ with $d$. Consider the set of all joins of such paths. It is a mail whose join will be the join of $C$. Conversely, in a chainmail every mail will have a join since every mail is mail-connected.
\end{proof}

For an element $x$ in a poset, the join of a non-empty subset of $x^\uparrow$ in the entire poset is the same as its join within $x^\uparrow$ (when one exists, the other exists and the two joins match). At the same time, the empty join in $x^\uparrow$ always exists and is given by $x$. So equivalently, a chainmail is a poset where $x^\uparrow$ is a complete lattice for every element $x$, i.e., when the dual poset is a locally complete category in the sense of \cite{freyd_scedrov_1990}. Locally completed preorders have recently been used in \cite{clementino_janelidze_2024}. Other than this, we are not aware of any literature on chainmails or their duals. Of course, chainmail orders are reminiscent of the well-known directed-complete partial orders (every directed set has a join) which arise in computer science. 

Note that a chainmail has smallest element if and only if it is a complete lattice.

\begin{example}\label{exaB}
A subset $C$ of a graph is connected when it is not empty and for any two elements in the subset, there is a path of edges connecting them, where all the vertices of the path are in $C$. Connected sets of a graph form a poset under subset inclusion. Now, a mail in this poset is a set of connected sets whose intersection contains a connected set, and hence at least one vertex $v$. The union of this set is then connected, since any two elements in the union can be connected with a path passing through $v$. This implies that the poset of connected subsets in a graph is a chainmail.
\end{example}

The example above has an application in chainmails. Firstly, mail-connected sets can be seen as connected sets for a graph structure where edges are two-element mails. Secondly, there is a more basic graph structure given by comparable elements: the union of the partial order with its opposite relation. Subsets of posets that are connected with respect to this graph structure we call \emph{connected sets} in the poset. Every connected set is clearly mail-connected. Furthermore, if $M$ is mail with a lower bound $b$, then the join of $M$ is the same as the join of $M\cup\{b\}$, which is clearly connected. So we have the following result.

\begin{theorem}
A poset is a chainmail if and only if every connected set in it has a join.    
\end{theorem}

\begin{example}\label{exaH}
Connectivity in a graph generalises to connectivity in a `hypergraph', i.e., a set $H$ equipped with a distinguished set $E$ of subsets called `hyperedges'. Define a non-empty subset $C$ of $H$ to be connected, when for any two elements $x,y\in C$ there is a sequence of hyperedges $E_0,\dots,E_n\in E$ such that: 
 (i) $x\in E_0$, $y\in E_n$, (ii) any two consecutive hyperedges in the sequence have nonempty intersection, and (iii) each hyperedge is a subset of $C$. 
The poset of connected sets is then a chainmail. To get connectivity in a graph, simply consider the hypergraph where the hyperedges are the singletons and two-element sets of vertices connected by an edge.    
\end{example}

Mails in a poset constitute a hypergraph structure on the poset. Connected sets with respect to this hypergraph structure are the same as mail-connected sets.

\begin{example}\label{exaC}
The poset of connected subsets of a topological space is a chainmail. This follows from the fact that the union of a mail of connected sets is connected. Indeed, suppose the union is a subset of the union of two disjoint open sets $A$ and $B$. Since there is a connected set which is a subset of every element of the mail, there is a point $x$ which is contained in every element of the mail. This point belongs either to $A$ or $B$. If it belongs to $A$, then every element of the mail will be forced to belong to $A$, by their connectivity. So in this case the union of the elements of the mail is a subset of $A$. Similarly, when the point belongs to $B$, the union of the elements of the mail is a subset of $B$. 
\end{example}

\begin{example}\label{exaD}
The union of a mail of path-connected sets in a topological space will be path-connected, since any two points in the union can be joined by a path that passes through an element of the lower bound of elements of the mail. Thus, path-connected sets in a topological space form a chainmail.
\end{example}

Every poset decomposes into a disjoint union of maximal connected subsets. They are called \emph{connected components} of the poset. It is easy to see that these happen to be the same as \emph{mail-connected components} (i.e., maximal mail-connected subsets).

A non-empty set $S$ of elements of a poset is said to be \emph{totally mail-disconnected} if no two distinct elements of $S$ have a common lower bound, or equivalently, its mail-connected subsets are only the singletons. Every singleton is both mail-connected and totally mail-disconnected, and moreover, singletons are the only such sets. We easily get the following result.

\begin{theorem}\label{thmL}
Every chainmail is a disjoint union of connected chainmails. These are the connected components of the chainmail, each of which has largest element. These largest elements are the maximal elements of the chainmail and they constitute a totally mail-disconnected set.
\end{theorem}

In Figure~\ref{fig:chainmail_table}, we show a computer-generated catalog of all connected chainmails having at most $7$ elements. See also Figure~\ref{tab:chml_counts} for the number of such chainmails per number $n$ of vertices for $n=0,\dots,10$, where the graph is given in the logarithmic scale. The sequence of these numbers was not previously listed in the Online Encyclopedia of Integer Sequences but has since been added by the first author \cite{OEIS:A374073}.

\begin{figure}
    \centering
    \includegraphics[width=0.98\textwidth]{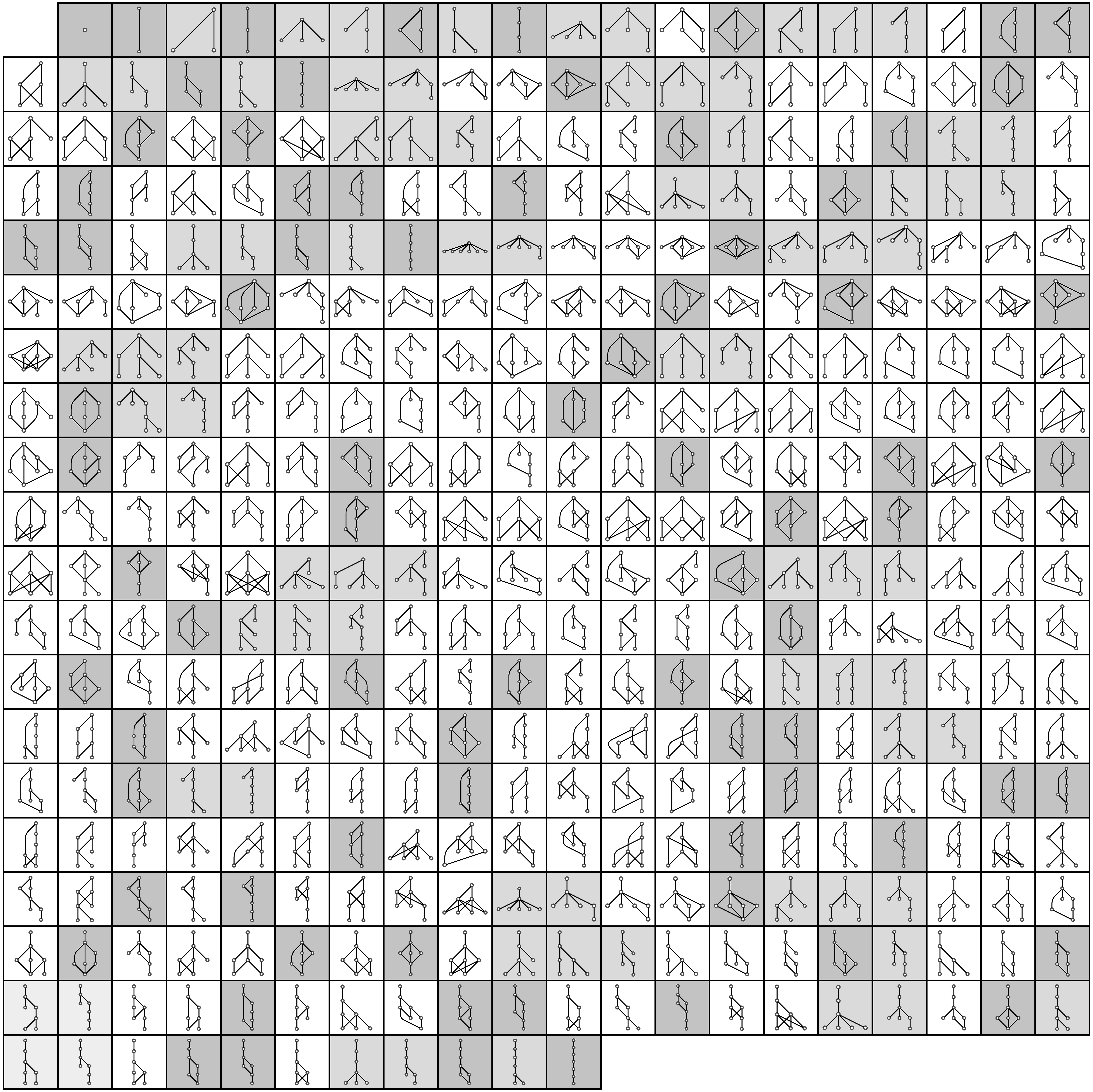}
    \caption{All connected chainmails having at most 7 elements. Lightly shaded cells represent chainmails that are rooted trees (Example~\ref{exaG}), except those that are chains. Darkly shaded cells represent those chainmails that are complete lattices.}
    \label{fig:chainmail_table}
\end{figure}

\pgfplotsset{compat=1.18}

\begin{figure}[ht]
    \centering
    \begin{tikzpicture}
    \begin{axis}[
        xlabel={},
        ylabel={},
        ymin=0,
        xmin=0,
        xmax=10,
        width=10cm,
        height=6cm,
        xtick={0,1,2,3,4,5,6,7,8,9,10},
        yticklabel style={/pgf/number format/fixed},
        scaled y ticks=false,
        ymajorgrids=true,
        grid style=dashed,
        log basis y={10},
        ymode=log, 
    ]

    \addplot[
        color=black,
        mark=*,
    ] coordinates {
        (0, 1)
        (1, 1)
        (2, 1)
        (3, 2)
        (4, 5)
        (5, 16)
        (6, 62)
        (7, 303)
        (8, 1842)
        (9, 14073)
        (10, 134802)
    };

    \end{axis}
\end{tikzpicture}
    \begin{tabular}{|r||r|r|r|r|r|r|r|r|r|r|r|}
        \hline
        $n=$ & 0 & 1 & 2 & 3 & 4 & 5 & 6 & 7 & 8 & 9 & 10\\\hline
        Mail-connected chainmails & 1 & 1 & 1 & 2 & 5 & 16 & 62 & 303 & 1842 & 14073 & 134802\\\hline
    \end{tabular}
    \caption{Number of mail-connected (single-component) chainmails having $n$ elements (the graph is logarithmic)}
    \label{tab:chml_counts}
\end{figure}

By a \emph{subposet} $X$ of a poset $Y$ we mean a subset $X$ of $Y$ equipped with the \emph{induced} partial order, i.e., partial order obtained on $X$ by restricting the partial order of $Y$ to $X$. The subposets of a partially ordered set themselves form a partially ordered set under inclusion of underlying sets. We will often not distinguish between the subposets of a poset $Y$ and their underlying sets (i.e., subsets of $Y$).

\begin{definition}
Call a subposet $\Gamma'$ of a chainmail $\Gamma$ a \emph{subchainmail} of $\Gamma$ when it is closed in $\Gamma$ under joins of mail-connected subsets of the subposet $\Gamma'$.
\end{definition}

A subchainmail of a chainmail is clearly itself a chainmail. By an adaptation of the argument used to prove Theorem~\ref{thmB}, we have:

\begin{theorem}\label{thmAA}
A subposet $\Gamma'$ of a chainmail $\Gamma$ is its subchainmail if and only if $\Gamma'$ is closed in $\Gamma$ under joins of mails in the subposet $\Gamma'$.
\end{theorem}

Chainmails are supposed to represent posets of connected elements in some ambient poset of regions. We now describe a construction that attempts to reconstruct this ambient poset.   

\begin{definition}
For a poset $\Gamma$, we write $\mathcal{D}(\Gamma)$ to denote the poset of totally mail-disconnected sets in $\Gamma$: $D_1\leqslant D_2$ for two totally mail-disconnected sets $D_1$ and $D_2$, whenever each element of $D_1$ is below some element of $D_2$. We call $\mathcal{D}(\Gamma)$ the \emph{exterior} of the poset $\Gamma$.
\end{definition}

The intuition here is that non-connected regions should admit a decomposition into connected components and thus, they can be represented by totally mail-disconnected sets of connected regions.


\begin{theorem}\label{thmF} For a down-closed set $X$ in a chainmail $\Gamma$, the following conditions are equivalent:
\begin{enumerate}
\item $X$ is a down-set of a totally mail-disconnected set in $\Gamma$.

\item $X$ equipped with the induced partial order is a subchainmail of $\Gamma$.
\end{enumerate}
Furthermore, the exterior of $\Gamma$ is isomorphic to the poset of down-closed subchainmails of $\Gamma$, i.e., subchainmails of $\Gamma$ whose underlying sets are down-closed in $\Gamma$. The isomorphism is given by mapping a totally mail-disconnected set to its down-set with induced partial order, and backwards, by mapping a down-closed subchainmail $\Gamma'$ to the set of joins of maximal mail-connected subsets of $\Gamma'$.
\end{theorem}

\begin{proof}
(1)$\Rightarrow$(2): Suppose $X$ is a down-set of totally mail-disconnected set $D$ in $\Gamma$. Consider a subset $M$ of $X$ that is a mail, and let $q$ be a lower bound of $M$. Each element of $M$ must be below one of the elements in $D$. This would force $q$ to be below all such elements of $D$, and so those must be the same element since $D$ is totally mail-disconnected. The join of $M$ will then also be below that element, which means that it is contained in $X$. Applying Theorem~\ref{thmAA} we get that $X$, equipped with induced partial order, is a subchainmail of $\Gamma$.

(2)$\Rightarrow$(1): Suppose $X$ is as in (2). Consider the maximal mail-connected subsets of $X$. The join of each such subset is then the top element in it. Furthermore, the set $D$ of these joins must be totally mail-disconnected. $X$ is then the down-set of $D$.

The last two statements of the theorem follow easily from the reduction just described of an $X$ satisfying (2) into a totally mail-disconnected set.
\end{proof}

\begin{theorem}\label{corB}
For a poset $\Gamma$, the exterior $\mathcal{D}(\Gamma)$ is a complete lattice if and only if $\Gamma$ is a chainmail.
\end{theorem}

\begin{proof}
Suppose $\mathcal{D}(\Gamma)$ is a complete lattice. Consider a mail $X$ in $\Gamma$. Let $Y$ be the join of $\{\{x\}\mid x\in X\}$ in $\mathcal{D}(\Gamma)$. Then each $x\in X$ is below some element $y_x$ of $Y$. Since $X$ has a lower bound, it is not possible for $y_x\neq y_{x'}$. It is then easy to see that $Y=\{y_x\}$, which also easily implies that $y_x$ must be the join of $X$ in $\Gamma$. So $\Gamma$ is a chainmail by Theorem~\ref{thmB}. Suppose now $\Gamma$ is a chainmail. It is clear that arbitrary intersection of subchainmails is a subchainmail. So the poset of subchainmails of $\Gamma$ is a complete lattice and hence $\mathcal{D}(\Gamma)$ is a complete lattice by Theorem~\ref{thmF}.
\end{proof}

In the next section we study the relation between the exterior of the chainmail of connected elements in some ambient lattice and that lattice. This relation will end up being a classifying tool for the concept of connectivity in a general lattice.

\section{Typical Connectivity Lattices}

A \emph{connectivity space} in the sense of R. ~B\"orger \cite{borger_1983} is a pair $(X,\mathcal{C})$ where $X$ is a set and $\mathcal{C}\subseteq\powerset{X}$ is a set of subsets of $X$ such that:
\begin{enumerate}[label=(CS\arabic*)]
\setcounter{enumi}{-1}
\item \label{cond:CS0} $\varnothing\notin \mathcal{C}$,
\item \label{cond:CS1} For any non-empty subset $Z\subseteq \mathcal{C}$ and $\bigcap Z\neq\varnothing$ implies $\bigcup Z\in \mathcal{C}$.
\item \label{cond:CS2} $\{x\}\in\mathcal{C}$ for any $x\in X$.
\end{enumerate}
Dugowson drops \ref{cond:CS2} in \cite{dugowson2010connectivity}. The same is true in \cite{stadler_stadler_2015}, where, in addition, \ref{cond:CS0} is replaced by its opposite: $\varnothing\in \mathcal{C}$. Note that each of \ref{cond:CS1} and \ref{cond:CS2} required for $\mathcal{C}$ is equivalent to requiring it for $\mathcal{C}^+=\mathcal{C}\setminus\{\varnothing\}$. The set $\mathcal{C}$, whose elements are called \emph{connected sets}, is a chainmail under subset inclusion: axioms \ref{cond:CS0} and \ref{cond:CS1} imply that the union of a mail of connected sets is connected, and hence it is the join of the mail in the poset of connected sets. 

Note that \ref{cond:CS2} has the following equivalent form:
\begin{enumerate}[label=(CS\arabic*$'$)]
\setcounter{enumi}{1}
\item \label{cond:CS2'} For any $Y\subseteq X$ there exists $Z\subseteq\mathcal{C}$ such that $Y=\bigcup Z$.
\end{enumerate} Examples~\ref{exaB} (connectivity in a graph), \ref{exaC} (topological connectivity) and \ref{exaD} (path-connectivity) of chainmails considered in the previous section arise in this way from a suitable connectivity space: in each case, take $X$ to be the underlying set of the ambient structure and let $\mathcal{C}$ be the set of connected subsets. For Example~\ref{exaH} (connectivity in a hypergraph), \ref{cond:CS0} and \ref{cond:CS1} still hold, but \ref{cond:CS2} fails. Below is another such example.

\begin{example}\label{exaG} It is not difficult to verify that for a finite poset $X$, the following conditions are equivalent:
\begin{enumerate}[label=(\roman*)]
    \item For every element $x$ of $X$ there is exactly one element $y$ of $X$ such that $x\leqslant y$ and $\{z\mid x<z<y\}=\varnothing$.
    \item Every mail in $X$ is linearly ordered.
    \item For every element $x$ in $X$, the mail $x^\uparrow$ is linearly ordered.
    \item The Hasse diagram of $X$ is a disjoint union of downward-growing rooted trees.
\end{enumerate}
Since a finite linearly ordered subset of any poset always admits a join (its largest element), a poset satisfying these equivalent conditions is trivially a chainmail. Moreover, the set
$$\mathcal{C}=\{x^\downarrow\mid x\in X\}$$
satisfies \ref{cond:CS0} and \ref{cond:CS1}, but fails to satisfy \ref{cond:CS2}.
\end{example}

The next example shows that there are chainmails that are not isomorphic to the chainmail of connected sets in any connectivity space. 

\begin{example}\label{exaA}
Consider the following poset:
\[\centering \vcenter{\scalebox{0.8}{\begin{tikzpicture}[
    dot/.style={circle, draw=black, fill=black, minimum size=0em, inner sep=0.15em}]
    
    \node[dot,label=right:1] (1) {};
    \node[dot,label=right:3] (3) [above=of 1]{};
    \node[dot,label=right:2] (2) [left=of 3]{};
    \node[dot,label=right:4] (4) [right=of 3]{};
    \node[dot,label=right:5] (5) [above=of 3]{};
    \node[dot,label=right:6] (6) [above=of 4]{};
    \node[dot,label=right:7] (7) [above=of 5]{};
    
    \draw[-] (1)--(2)--(5)--(7);
    \draw[-] (1)--(3)--(5);
    \draw[-] (3)--(6)--(7);
    \draw[-] (4)--(5);
    \draw[-] (4)--(6);
    \end{tikzpicture}}}\]
Since taking away an element from a set of elements in a poset which is below another element in the same set does not affect the join of the set, to check that the poset above is a chainmail it is sufficient to check that joins exist for all \emph{reduced mails}, i.e., mails where no two elements are comparable. Apart from the singletons and the mail containing the top element, in which case joins trivially exist, reduced mails along with their joins are: \[\bigvee \{2,3\}=5, \bigvee \{2,6\}=7, \bigvee \{5,6\}=7.\]
Thus, the poset above is indeed a chainmail. This chainmail cannot be a chainmail of connected sets in any connectivity space, even one satisfying only \ref{cond:CS1}. Indeed, the connected set $4$, being a subset of the union of $2$ and $3$, must contain an element from $3$ (otherwise it would have been a subset of $2$). But then $3$ and $4$ must have a join in the chainmail, which they do not.
\end{example}

According to Definition~3 in \cite{Serra1998Nov}, a subset $\mathcal{C}\subseteq L$ of a complete lattice $L$ is said to define a \emph{connection} if it has the following properties:
\begin{enumerate}[label=(CL\arabic*)]
\setcounter{enumi}{-1}
\item \label{cond:CL0} $0\in \mathcal{C}$. 

\item \label{cond:CL1} The join of every non-empty subset $X\subseteq \mathcal{C}$ that has a lower bound in $L$ different from $0$, belongs to $\mathcal{C}$. 

\item \label{cond:CL2} Every element $a\in L$ is a join of elements from $\mathcal{C}$.
\end{enumerate}
Note that \ref{cond:CL1} and \ref{cond:CL2} are pointfree counterparts of \ref{cond:CS1} and \ref{cond:CS2'}.
As in the point-set case, \ref{cond:CL1} implies that $\mathcal{C}^+=\mathcal{C}\setminus\{0\}$ is a chainmail. Note that in particular, $\mathcal{C}=L$ defines a connection and $L^+=L\setminus\{0\}$ is a chainmail for any complete lattice $L$. As in the point-set case, \ref{cond:CL0} does not affect \ref{cond:CL1} and \ref{cond:CL2}.

See \cite{Serra1998Nov, braga2003connectivity} for examples of connections (and therefore of chainmails) that are relevant to the analysis of digital images. In this field, transitioning from a point-set approach to connectedness (e.g., one given by a connectivity space) to an abstract lattice-based approach, given by connections, becomes essential when analyzing non-binary images, where each pixel can take on more than two colors. While binary images can be fully characterized by the set of pixels sharing one of two possible colors, the same is no longer true for images containing a broader color palette.

In what follows we will unpack the notion of a connection. We begin with the following observation.

\begin{theorem}\label{thmW}
For a subset $\mathcal{C}$ of a complete lattice $L$, \ref{cond:CL1} is equivalent to the following condition:
\begin{itemize}
\item[(CL1$'$)] For any two elements $x\leqslant y$ where $x\neq 0$ in $L$, the set $$[x,y]\cap\mathcal{C}=\{c\in \mathcal{C}\mid x\leqslant c\leqslant y\}$$ contains a largest element if it is non-empty.
\end{itemize}
\end{theorem}

\begin{proof}
Suppose \ref{cond:CL1} holds. Consider two elements $x\leqslant y$ of $L$ with $[x,y]\cap\mathcal{C}$ non-empty. By \ref{cond:CL1}, the join $\bigvee [x,y]\cap\mathcal{C}$ belongs to $\mathcal{C}$ and hence is the maximal element in $\bigvee [x,y]\cap\mathcal{C}$. Now suppose (CL1$'$) holds. Consider a non-empty subset $X\subseteq \mathcal{C}$ having a lower bound $x$ different from $0$. For each $z\in X$, let $c_z$ be the unique maximal element of $[z,\bigvee X]\cap\mathcal{C}$. Then each $c_z$ is also the unique maximal element of $[x,\bigvee X]\cap\mathcal{C}$. So all $c_z$'s are equal to some $c\in \mathcal{C}$. Since $c\leqslant \bigvee X$ and each element of $X$ is below $c$, we get that $\bigvee X=c$, thus proving that $\bigvee X\in\mathcal{C}$.
\end{proof}

We combine the approach of B\"orger to require the opposite of \ref{cond:CL0} and that of Dugowson from \cite{dugowson2010connectivity} to drop \ref{cond:CS2} and introduce the following concept.

\begin{definition}
A \emph{typical connectivity lattice} is a pair $(L,\mathcal{C})$ where $L$ is a complete lattice and $\mathcal{C}$ is a subset of $L$ such that \ref{cond:CL1} holds, but \ref{cond:CL0} does not hold; \emph{connected} elements of a typical connectivity lattice $(L,\mathcal{C})$ are the elements of $\mathcal{C}$. When $(L,\mathcal{C})$ is a typical connectivity lattice, the subset $\mathcal{C}$ of $L$ is called a \emph{typical connectivity} in $L$.
\end{definition}

The following result gives Theorem~1.4(i) in \cite{borger_1983} in the case of connectivity spaces. We omit the proof as it is an easy adaptation of the proof of Theorem~\ref{thmB}.

\begin{theorem}\label{lemC} In a typical connectivity lattice $(L,\mathcal{C})$, the join of a mail-connected (in $L^+$) set $C$ of connected elements is connected.
\end{theorem}

In any typical connectivity lattice $(L,\mathcal{C})$, we can partition the set of connected elements below an element $x\in L$ in such a way that each member $C$ of the partition is a maximal mail-connected set in $L^+$. Consider the join $c$ of each member $C$ of such partition. By the theorem above, $c$ will itself be connected and hence $c\in C$. We denote the set of such $c$'s by $\mathcal{C}(x)$. We call its elements \emph{$\mathcal{C}$-components} of $x$ (sometimes, just \emph{components} or \emph{connected components}). Note that $\mathcal{C}(x)\subseteq L^+$.

\begin{theorem}\label{lemB}\label{lem:x_star_general}\label{thmJ}
For every typical connectivity lattice $(L,\mathcal{C})$, the map $$L\to \mathcal{D}(\mathcal{C}),\quad x\mapsto \mathcal{C}(x)$$ has a left Galois adjoint given by
$S\mapsto \bigvee S.$ Furthermore, for any element $x\in L$, the set $\mathcal{C}(x)$ is totally mail-disconnected in $L^+$. Moreover, $x=\bigvee \mathcal{C}(x)$ if and only if $x$ is a join of connected elements of $L$.
\end{theorem}

\begin{proof} Let $S\in\mathcal{D}(\mathcal{C})$ and $x\in L$. If $\bigvee S\leqslant x$ then every element of $S$ is below $x$ and hence below one of the elements of $\mathcal{C}(x)$, showing $S\leqslant \mathcal{C}(x)$. If $S\leqslant \mathcal{C}(x)$ then $\bigvee S\leqslant \bigvee \mathcal{C}(x)\leqslant x$. This proves that we indeed have the claimed Galois connection. That each $\mathcal{C}(x)$ is totally mail-disconnected in $L^+$ (and hence also in $\mathcal{C}$) is obvious from the definition of $\mathcal{C}(x)$. Note that every element of $\mathcal{C}(x)$ is below $x$. Since every connected element below $x$ is below one of the elements of $\mathcal{C}(x)$, if $x$ is a join of connected elements, then $x=\bigvee \mathcal{C}(x)$. Conversely, if  $x=\bigvee\mathcal{C}(x)$ then clearly $x$ is a join of connected elements.
\end{proof}

It turns out that the theorem above has a converse, and hence it provides a characterization of typical connectivity:

\begin{theorem}\label{thmK}
For any complete lattice $L$ and any subset $\mathcal{C}\subseteq L$, if the map $S\mapsto \bigvee S$ from the poset $\mathcal{D}(\mathcal{C})$ of totally mail-disconnected subsets of $\mathcal{C}$ to $L$ has a right Galois adjoint $x\mapsto C_\mathcal{C}(x)$, and moreover, each $C_\mathcal{C}(x)$ is a totally mail-disconnected set in $L^+$, then $0\notin \mathcal{C}$ and $\mathcal{C}$ is a typical connectivity in $L$ with $C_\mathcal{C}(x)=\mathcal{C}(x)$ for every $x\in L$. 
\end{theorem}

\begin{proof}
First, we show that $0\notin \mathcal{C}$. If $0\in \mathcal{C}$, then since $\{0\}\in\mathcal{D}(\mathcal{C})$ and $\bigvee \{0\}=0\leqslant 0$, we must have $\{0\}\leqslant C_\mathcal{C}(0)$, which, together with $\bigvee C_\mathcal{C}(0)\leqslant 0$, would force $0\in C_\mathcal{C}(0)$. However, $C_\mathcal{C}(0)\subseteq L^+$, hence a contradiction. Next, we show that $\mathcal{C}$ satisfies \ref{cond:CL1}. Let $X\subseteq\mathcal{C}$ be such that $X$ has a lower bound in $L^+$. For each $x\in X$ we have $\{x\}\in\mathcal{D}(\mathcal{C})$. Also, $\{x\}=\bigvee\{x\}\leqslant \bigvee X$ and hence $\{x\}\leqslant C_\mathcal{C}(\bigvee X)$. This means that $x\leqslant x'\in C_\mathcal{C}(\bigvee X)$ for some $x'$. Since $C_\mathcal{C}(\bigvee X)$ is a totally mail-disconnected set in $L^+$ and $X$ has a lower bound in $L^+$, we must have that $x'=y'$ for any two $x,y\in X$. But then, $\bigvee X\leqslant x'$. At the same time, $x'\leqslant \bigvee C_\mathcal{C}(\bigvee X)\leqslant \bigvee X$ and so $\bigvee X=x'\in \mathcal{C}$. Since the maps $C_\mathcal{C}(-)$ and $\mathcal{C}(-)$ are both right Galois adjoints to the same map, they must be equal.
\end{proof}

The two theorems above provide the following fundamental outlook on typical connectivity: \emph{a typical connectivity is a class of connected elements in a complete lattice that admits a well-behaved decomposition of elements into connected components}. Theorem~\ref{thmJ} above also shows that the concept of connection in the sense of J.~Serra, up to discarding $0$, arises as a special case of a typical connectivity where every element can be reconstructed back from its connected components (as their join): so the decomposition of elements into connected components is even better behaved than in the case of typical connectivity. 

\begin{definition}
A \emph{Serra connectivity} in a complete lattice $L$ is a typical connectivity $\mathcal{C}$ in $L$ satisfying \ref{cond:CL2}. The corresponding pair $(L,\mathcal{C})$ is called a Serra connectivity lattice.
\end{definition}

The precise relation between Serra connectivity and J.~Serra's concept of a connection is the following: A subset $\mathcal{C}$ of a complete lattice $L$ is a connection if and only if $0\in\mathcal{C}$ and $\mathcal{C}^+$ is a Serra connectivity in the sense of the definition above.

\section{General Connectivity Lattices}

Consider a complete lattice $L$ and a subset $\mathcal{C}\subseteq L$. For each $x\in L$, we write $\mathcal{C}(x)$ for the set of maximal elements in $\mathcal{C}\cap x^\downarrow$ and call its elements \emph{$\mathcal{C}$-components} of $x$. The join $\bigvee\mathcal{C}(x)$ in $L$ is called the \emph{kernel} of $x$. 

\begin{theorem}\label{thmM}
Let $L$ be a complete lattice. For a subset $\mathcal{C}\subseteq L$, the map $\mathcal{D}(\mathcal{C})\to L$ given by $S\mapsto\bigvee S$ has a right Galois adjoint, if and only if $\mathcal{C}$ equipped with the induced partial order is a subchainmail of $L$.
\end{theorem}

\begin{proof}
The proof of this theorem uses arguments adapted from the proofs of Theorems~\ref{thmJ} and \ref{thmK}. In fact, the proof of the `only if' part is almost identical to the middle part of the proof of Theorem~\ref{thmK} --- just replace $L^+$ in it with $\mathcal{C}$. We prove the `if' part and the last part of the theorem simultaneously. Suppose $\mathcal{C}$ is a subchainmail of $L$. Then it is easy to see that $\mathcal{C}$ is a chainmail. The poset $\mathcal{C}\cap x^\downarrow$ is clearly a down-closed subchainmail of $\mathcal{C}$. Let $\mathcal{C}(x)$ be the set of maximal elements of this subchainmail. Then each $\mathcal{C}(x)$ is an element of $\mathcal{D}(\mathcal{C})$, by Theorem~\ref{thmL}. The rest of the proof is identical to the first few lines of the proof of Theorem~\ref{thmJ}.
\end{proof}

\begin{example}\label{exaN}
In general, even if $\mathcal{C}\subseteq L$ satisfies \ref{cond:CL2}, it may still be a chainmail without being a subchainmail of $L$. This is the case, for instance, in the following lattice,
        \[{\scalebox{0.8}{\begin{tikzpicture}[
        dot/.style={circle, draw=black, fill=white, minimum size=0em, inner sep=0.15em}, filleddot/.style={circle, draw=black, fill=black, minimum size=0em, inner sep=0.15em}, node distance=0.9cm]

        \node[filleddot] (1) {};
        \node[dot] (2) [above=of 1]{};
        \node[dot] (3) [above left=of 2]{};
        \node[dot] (4) [above right=of 2]{};
        \node[filleddot] (5) [above right=of 3]{};
        \node[dot] (6) [above=of 5]{};
        
        \draw[-] (1)--(2)--(3)--(5)--(6);
        \draw[-] (2)--(4)--(5);
    \end{tikzpicture}}}\]
where $\mathcal{C}$ is given by the set of hollow nodes.
\end{example}

\begin{example}\label{exaO}
Consider the lattice $L$ of all subsets of a unit square $[0,1]\times[0,1]$ and let $\mathcal{C}$ be the subset of $L$ consisting of all rectangular subsets, i.e., sets of the form $[a,b]\times[c,d]$ where $a,b,c,d\in[0,1]$ with $a\leqslant b$ and $c\leqslant d$. It is not difficult to see that $\mathcal{C}$ has all non-empty joins. Although $\mathcal{C}$ is not a subchainmail of $L$ (union of overlapping rectangles is not a rectangle), $\mathcal{C}$ is still a chainmail under the induced partial order. Furthermore, $\mathcal{C}$ satisfies \ref{cond:CL2}. 
\end{example}

\begin{definition} A \emph{preconnectivity lattice} is a pair $(L,\mathcal{C})$ where $L$ is a complete lattice and $\mathcal{C}$ is a subset of $L$ that is a chainmail under the induced partial order. Elements of $\mathcal{C}$ are called \emph{connected elements} of the preconnectivity lattice, while elements of $\mathcal{C}(x)$ are called \emph{connected components} of $x$. When $\mathcal{C}$ is a subchainmail of $L$, we call a preconnectivity lattice a \emph{connectivity} lattice; we will refer to the corresponding Galois connection of Theorem~\ref{thmM} as the \emph{connectivity adjunction}. The set $\mathcal{C}$ is a \emph{(pre)connectivity} in $L$ when $(L,\mathcal{C})$ is a \emph{(pre)connectivity} lattice.
\end{definition}

\begin{example}\label{exaV}
For a complete lattice $L$, any set $\mathcal{C}\subseteq L^+$ in which any two elements are mutually incomparable is a connectivity. It is easy to see that such connectivity is not, in general, a typical connectivity.
\end{example}

It is clear from the results of the previous section that each, a typical connectivity as well as a Serra connectivity, is a connectivity. Ahead, we will have further adjectives that we place in front of `connectivity', to describe special types of connectivities. When introducing these terms, we will automatically assume that whenever `type X connectivity' is defined as a set $\mathcal{C}$ of elements in a complete lattice $L$, the corresponding pair $(L,\mathcal{C})$ will be referred to as `type X connectivity lattice' (and vice versa).

Theorem~\ref{thmW} has the following analogue.

\begin{theorem}\label{thmV}
A subset $\mathcal{C}$ of a complete lattice $L$ is a connectivity if and only if for every element $x\in L$, each element of $\mathcal{C}\cap x^\downarrow$ is below a unique $\mathcal{C}$-component of $x$.
\end{theorem}

Next, we analyze the condition \ref{cond:CL0} in the context of connectivity lattices. 

\begin{theorem}\label{thmT}
For a connectivity lattice $(L,\mathcal{C})$ the following conditions are equivalent:
\begin{itemize}
\item[(i)] $0\in\mathcal{C}$.

\item[(ii)] $\mathcal{C}$ is closed under all joins in $L$.

\item[(iii)] $\mathcal{C}(x)=\{\bigvee(\mathcal{C}\cap x^\downarrow)\}$.

\item[(iv)] Every element $x\in L$ has exactly one component.

\item[(v)] The kernel of every element $x\in L$ is connected.
\end{itemize}
When these conditions hold, $\mathcal{D}(\mathcal{C})=\{\{x\}\mid x\in \mathcal{C}\}\cup\{\varnothing\}$.
\end{theorem}

\begin{proof}
(i)$\Rightarrow$(ii): Suppose $0\in\mathcal{C}$. Then any non-empty set in $\mathcal{C}$ is a mail and so $\mathcal{C}$ is closed in $L$ under all non-empty joins. Since $0\in\mathcal{C}$, we also have closure under the empty join. 

By Theorem~\ref{thmM}, each $\mathcal{C}(x)$ is the set of maximal elements in $\mathcal{C}\cap x^\downarrow$. It is the clear that (ii)$\Rightarrow$(iii).

(iii)$\Rightarrow$(iv) is obvious.

(iv)$\Rightarrow$(v): If $x$ has exactly one component, then that component will be the kernel of $x$.

(v)$\Rightarrow$(i), since the kernel of $0$ is $0$.

Clearly, when $0\in\mathcal{C}$, the only totally mail-disconnected sets in the chainmail $\mathcal{C}$ are the empty set and the singletons of connected elements.
\end{proof}

\begin{definition}
We call connectivity lattices $(L,\mathcal{C})$ where $\mathcal{C}$ satisfies \ref{cond:CL0}, \emph{kernel connectivity} lattices.
\end{definition}

As the following example shows, open sets in a topological space can in a sense be viewed as `connected sets'.

\begin{example}\label{exaI}
For an open-set topology $\tau$ on a set $X$, the pair $(\mathcal{P}(X),\tau)$ is a kernel connectivity lattice. In this case, the kernel of a subset of $X$ is its topological interior. Both \ref{cond:CL0} and \ref{cond:CL1} hold, but \ref{cond:CL2} fails in general.   
\end{example}

\begin{theorem}
For a connectivity lattice $(L,\mathcal{C})$, the condition \ref{cond:CL1} holds if and only if each $\mathcal{C}(x)\setminus\{0\}$ is totally mail-disconnected in $L^+$. Furthermore, \ref{cond:CL0} implies \ref{cond:CL1}.
\end{theorem}

\begin{proof}
Suppose \ref{cond:CL1} holds. If $0\notin\mathcal{C}$,  then by Theorem~\ref{thmJ}, $\mathcal{C}(x)=\mathcal{C}(x)\setminus\{0\}$ is totally mail-disconnected in $L^+$. The converse is also true thanks to Theorems~\ref{thmK} and \ref{thmM}. If $0\in\mathcal{C}$, then by Theorem~\ref{thmT}, $\mathcal{C}(x)\setminus\{0\}$ is either empty, or a singleton subset of $L^+$. In both cases, it is totally mail-disconnected in $L^+$. At the same time, by Theorem~\ref{thmT} again, \ref{cond:CL1} clearly holds (so, in particular, \ref{cond:CL0} implies \ref{cond:CL1}).    
\end{proof}

Here is an example of a connectivity that is neither a kernel connectivity nor a  typical connectivity.

\begin{example}\label{exaJ}
    A connected subset in a graph is $k$-connected if it remains connected (and hence nonempty) after removing up to $k-1$ vertices. The powerset lattice of the vertices of the graph together with the set of $k$-connected sets is a connectivity lattice. This is analogous to the example of connected sets in a graph (Example~\ref{exaB}), but a mail is now a set of $k$-connected sets whose intersection contains a $k$-connected set $S$. If we remove up to $k-1$ vertices from the union of this mail, then it is still connected because the sets in the mail remain connected and the set $S$ remains nonempty and it therefore allows for paths between them. In the graph below, 
        $$\centering \vcenter{\scalebox{0.8}{\begin{tikzpicture}[
        dot/.style={circle, draw=black, fill=black, minimum size=0em, inner sep=0.15em}, node distance=0.9cm]
        
        \node[dot] (1) {};
        \node[dot] (3) [above right=of 1]{};
        \node[dot] (2) [below right=of 1]{};
        \node[dot] (4) [above right=of 2]{};
        \node[dot] (5) [above right=of 4]{};
        \node[dot] (6) [below right=of 4]{};
        \node[dot] (7) [above right=of 6]{};
        
        \draw[-] (1)--(2)--(4)--(3)--(1);
        \draw[-] (4)--(5)--(7)--(6)--(4);
    \end{tikzpicture}}}$$
    there are two $2$-connected sets: the two diamonds. Even though they have non-empty intersection, their union is not $2$-connected. So this gives an example of a connectivity lattice where none of \ref{cond:CL0} and \ref{cond:CL1} hold.
\end{example}

Below is another such example.

\begin{example}\label{exaM}
Let $X$ be a set and $Y$ its subset. Consider the power set $L$ of $X$. Let $\mathcal{C}$ be the set of those subsets $C$ of $X$ such that $C\setminus Y$ is a singleton. It is not difficult to see that the pair $(L,\mathcal{C})$ will be a connectivity lattice. \ref{cond:CL0} clearly fails in this case. As soon as $Y$ has at least one element and $X\setminus Y$ has at least two elements, \ref{cond:CL1} will not hold either. For a subset $A$ of $X$, its components are given by $$\mathcal{C}(A)=\{\{a\}\cup (A\cap Y)\mid a\in A\cap (X\setminus Y)\}.$$
\end{example}

\begin{theorem}\label{thmU} 
For a connectivity lattice $(L,\mathcal{C})$, the condition \ref{cond:CL2} holds if and only if every element $x\in L$ is the kernel of itself.
\end{theorem}

\begin{proof}
If \ref{cond:CL2} holds, then $x$ is the join of connected elements below $x$. Each of these connected elements is below one of the components (by Theorem~\ref{thmV}). This forces $x=\bigvee\mathcal{C}(x)$. Conversely, if $x=\bigvee\mathcal{C}(x)$ for each $x\in L$ then \ref{cond:CL2} holds since an element of $\mathcal{C}(x)$ is always connected.
\end{proof}

\begin{theorem}\label{thmS}
For a subset $\mathcal{C}$ of a complete lattice $L$, if either \ref{cond:CL0} or \ref{cond:CL2} holds, then the following also holds:
\begin{enumerate}[label=(CL\arabic*½)]
    \item \label{cond:CL1.5} For every element $a\in L$, if $a\neq 0$ then there exists $c\in\mathcal{C}$ such that  $c\leqslant a$.
\end{enumerate} 
If $\mathcal{C}$ is a connectivity and \ref{cond:CL1.5} holds, then \ref{cond:CL1} holds as well. Furthermore, $\mathcal{C}=L$ if and only if: $\mathcal{C}$ is a connectivity and \ref{cond:CL0}, \ref{cond:CL2} hold.
\end{theorem}

\begin{proof}
That \ref{cond:CL0} and \ref{cond:CL2} each imply \ref{cond:CL1.5} is close to obvious. Under \ref{cond:CL1.5}, a set $X\subseteq\mathcal{C}$ is a mail in $\mathcal{C}$ provided it is a mail in $L^+$. This easily gives that if $\mathcal{C}$ is a connectivity then \ref{cond:CL1} holds. If $\mathcal{C}=L$ then it is clear that $\mathcal{C}$ is a connectivity and \ref{cond:CL0}, \ref{cond:CL2} hold. The converse statement is a consequence of Theorems~\ref{thmT} and \ref{thmU}.
\end{proof}

\begin{definition}
We call a connectivity lattice $(L,\mathcal{C})$ \emph{well-founded} when \ref{cond:CL1.5} holds and \emph{degenerate} when $\mathcal{C}=L$. We call $(L,\mathcal{C})$ a \emph{saturated connectivity lattice} when \ref{cond:CL2} holds.
\end{definition}

Theorem~\ref{thmS} gives the following.

\begin{corollary}\label{thmR}
A well-founded connectivity is a typical connectivity if and only if it is not a kernel connectivity. Consequently, a connectivity is a Serra connectivity if and only if it is not a kernel connectivity and is saturated.
\end{corollary}

\begin{example}\label{exaY}
For an arbitrary well-founded Serra connectivity $\mathcal{C}$ in a complete lattice $L$, the set $\mathcal{C}$ remains a typical well-founded connectivity in the extension $L\sqcup \{\infty\}$ with a new top element $\infty$. However, it is no longer a Serra connectivity.
\end{example}

By Theorem~\ref{thmU}, saturated connectivity is one for which the right adjoint in the connectivity adjunction is a right inverse of the left adjoint. Next, we examine the symmetric condition: when the right adjoint is a left inverse of the left adjoint. 

\begin{definition}
A connectivity lattice $(L,\mathcal{C})$ is said to be \emph{separated} when the following holds: 
\begin{itemize}
\item[(CL3)] For any totally mail-disconnected set $X$ in $\mathcal{C}$ we have $X=\mathcal{C}(\bigvee X)$.
\end{itemize}
\end{definition}
\begin{definition}
An \emph{absolute connectivity lattice} is a connectivity lattice in which the connectivity adjunction is an isomorphism.
\end{definition}

\begin{theorem}\label{thmQ}
A separated connectivity lattice is never a kernel connectivity lattice. Consequently: 
\begin{itemize}
\item[(a)] A well-founded separated connectivity lattice is a typical connectivity lattice.

\item[(b)] A separated Serra connectivity lattice is the same as an absolute connectivity lattice.  
\end{itemize}
\end{theorem}

\begin{proof}
In a kernel connectivity lattice $(L,\mathcal{C})$, the Galois closure of $\varnothing\in\mathcal{D}(\mathcal{C})$ is given by $\{0\}$, and so $(L,\mathcal{C})$ cannot be separated. Then, (a) follows from Corollary~\ref{thmR} and we also easily get (b).  
\end{proof}

\begin{example}\label{exaL}
The connectivity lattice $(L,\mathcal{C})$ from Example~\ref{exaC}, where $L$ is the powerset of the underlying set of a topological space and $\mathcal{C}$ is the set of connected sets in the space, while being a Serra connectivity lattice, is not separated. This is because the union of two disjoint connected sets may be connected. In fact, any connected set is the union of singletons, which are disjoint and connected. We can then see that separatedness of this connectivity lattice is equivalent to the space being totally disconnected.
\end{example}

\begin{example}\label{exaK} Let us consider a slight modification to Example~\ref{exaG}. Consider the lattice $L$ of all down-closed subsets of a poset $P$ of the form described in Example~\ref{exaG}. Let $\mathcal{C}$ be the set of all principal down-closed sets, i.e., sets of the form $x^\downarrow$ where $x\in P$. Then the pair $(L,\mathcal{C})$ is a separated Serra connectivity lattice.
\end{example}

\begin{example}\label{exaW}
In Example~\ref{exaV} (where $\mathcal{C}$ is given by an anti-chain), the set $\mathcal{C}$ is a separated connectivity if and only if there is no $x\in \mathcal{C}$ which is below the join of any $X\subseteq \mathcal{C}\setminus{x}$. For instance, such is any set $\mathcal{C}$ of non-empty subsets of a set $A$ (where the lattice $L$ is the powerset of $A$), as soon as each element of $\mathcal{C}$ contains an element not contained in any of the other elements of $\mathcal{C}$. For example, in the power set of $A=\{1,2,3\}$, the set $\mathcal{C}=\{\{1,2\},\{2,3\}\}$ is a separated connectivity. This separated connectivity is not a typical connectivity.
\end{example}

\begin{example}\label{exaX}
In Example~\ref{exaV} again, the set $\mathcal{C}$ is a well-founded connectivity if and only if every element different from $0$ has an element of $\mathcal{C}$ below it. This implies that the elements of $\mathcal{C}$ are exactly the atoms. An example of such connectivity is given by the three atoms in the following complete lattice:
\[\scalebox{0.8}{\begin{tikzpicture}[
        dot/.style={circle, draw=black, fill=black, minimum size=0em, inner sep=0.15em}, node distance=0.9cm]
    
    \node[dot] (1) {};
    \node[dot] (3) [above=of 1]{};
    \node[dot] (2) [left=of 3]{};
    \node[dot] (4) [right=of 3]{};
    \node[dot] (5) [above=of 3]{}; 
    [above=of 4]{};
    \node[dot] (7) [above=of 5]{};
    
    \draw[-] (1)--(2)--(5)--(7);
    \draw[-] (1)--(3)--(5);
    \draw[-] (4)--(5);
    \draw[-] (4)--(1);
    \end{tikzpicture}}\]
Since this is not a kernel connectivity, it is a typical connectivity. However, it is neither separated nor saturated.
\end{example}

\begin{example}\label{exaZ}
In Example~\ref{exaY}, if $\mathcal{C}$ were in addition a separated connectivity in the original complete lattice $L$, then it will remain such in the extension $L\sqcup \{\infty\}$ (the converse is also true). So a well-founded separated connectivity need not be saturated.
\end{example}

There is a general process behind Example~\ref{exaK}. Let $L$ be a complete lattice. For a subset $\mathcal{C}\subseteq L$, consider the set $\Sigma\mathcal{C}$ of all joins of subsets of $\mathcal{C}$ in $L$. Then it is not difficult to see that: 
\begin{itemize}
\item $(L,\mathcal{C})$ is a (separated) connectivity lattice if and only if $(\Sigma\mathcal{C},\mathcal{C})$ is a (separated) connectivity lattice.

\item $(L,\mathcal{C})$ is a kernel connectivity lattice if and only if $(\Sigma\mathcal{C},\mathcal{C})$ is a kernel connectivity lattice (in this case, $\Sigma\mathcal{C}=\mathcal{C}$).

\item \ref{cond:CL2} and hence \ref{cond:CL1.5} and \ref{cond:CL1} will always hold for $(\Sigma\mathcal{C},\mathcal{C})$ and so, the pair $(L,\mathcal{C})$ is a connectivity lattice with $0\neq\mathcal{C}$ if and only if $(\Sigma\mathcal{C},\mathcal{C})$ is a Serra connectivity lattice. 
\end{itemize}

\section{Absolute Connectivity Lattices}

Given a complete lattice $L$, consider the following conditions on an element $a$ in $L$.

\begin{enumerate}[label=(E\arabic*)]
    \item \label{cond:E1} $a\neq 0$ and for any $x,y$, if $a\leqslant x\vee y$ and $x\wedge y=0$, then $a\leqslant x$ or $a\leqslant y$.
    \item \label{cond:E2} $a\neq 0$ and for any $x,y$, we have:
    \[\big[x\wedge y=0\text{ and }a=x\vee y\big]\implies \big[x=a\text{ or } y=a\big].\]
    \item \label{cond:E3} for any totally mail-disconnected subset $S$ of $L^+$, if $a=\bigvee S$ then $a\in S$ (and consequently, $S=\{a\}$).
    \item \label{cond:E4} for any totally mail-disconnected subset $S$ of $L^+$, if $a\leqslant \bigvee S$ then $a\leqslant s$ for some $s\in S$.
\end{enumerate}

Note that if in \ref{cond:E1} and \ref{cond:E2} we drop `$x\wedge y=0$', then we get the conditions defining join-prime and join-irreducible elements, respectively (up to allowing $a=0$ as in, e.g., \cite{Balachandran1955}). Similarly, if in \ref{cond:E3} and \ref{cond:E4} we drop the requirement on $S$ to be a totally mail-disconnected subset of $L^+$, then we get the conditions defining completely join-irreducible and completely join-prime elements, respectively (see again \cite{Balachandran1955}). 

It is easy to see that we have the following implications:
$$\xymatrix@!=10pt{ & \textrm{(E2)} & \\ \textrm{(E1)}\ar@{=>}[ur] & & \textrm{(E3)}\ar@{=>}[ul] \\ & \textrm{(E4)}\ar@{=>}[ul]\ar@{=>}[ur] & }$$
Note that $0=\bigvee\varnothing$, where $\varnothing$ is a totally mail-disconnected subset of $L^+$. So $0$ cannot satisfy \ref{cond:E3} or \ref{cond:E4}. We say that an element $a$ in $L$ is \emph{absolutely connected} when \ref{cond:E4} holds, or equivalently, all of \hyperref[cond:E1]{(E1-4)} hold. It turns out that in a frame, these conditions are always equivalent.

\begin{theorem}\label{thmY}
An element $a$ in a frame is absolutely connected if and only if it satisfies any one of \hyperref[cond:E1]{(E1-4)}.   
\end{theorem}

\begin{proof} It suffices to show that \ref{cond:E2}$\Rightarrow$\ref{cond:E4}.  Assuming \ref{cond:E2}, suppose $a\leqslant \bigvee S$ where $S$ is totally mail-disconnected in $L^+$, i.e. $s_1\wedge s_2=0$ for any $s_1,s_2\in S$. Then, $$a=a\wedge\bigvee S=\bigvee\{a\wedge s\mid s\in S\}.$$
Since $a\neq 0$ by \ref{cond:E2}, there exists $b\in S$ such that $a\wedge b\neq 0$. We have:
\[(a\wedge b)\wedge \bigvee \{a\wedge s\mid s\in S\setminus\{b\}\}=\bigvee\{a\wedge b\wedge s\mid s\in S\setminus\{b\}\}=0\]
At the same time, we have:
\[a=(a\wedge b)\vee \bigvee \{a\wedge s\mid s\in S\setminus\{b\}\}.\]
Since $a\wedge b\neq 0$, by \ref{cond:E2} we have:
$$\bigvee \{a\wedge s\mid s\in S\setminus\{b\}\}=0.$$
Then $a=a\wedge b$ and so $a\leqslant b$.
\end{proof}

In general, however, the conditions \hyperref[cond:E1]{(E1-4)} need not be equivalent, as the following examples show.

\begin{example}\label{exaE}
In the complete lattice of closed sets of a topological space where every singleton is closed (e.g., a Hausdorff space), a closed set $a$ satisfies \ref{cond:E1} if and only if it satisfies \ref{cond:E2} and if and only if it is a connected closed set; $a$ satisfies \ref{cond:E3} if and only if it is a singleton. It satisfies \ref{cond:E4} if and only if it is a clopen singleton.
\end{example}

\begin{example}\label{exaF}
    For the lattices $M_3$ and $N_5$ (Figure \ref{fig:e1_e2_different}), in each case we have that the element $a$ satisfies \ref{cond:E2} but not \ref{cond:E1}.
\end{example}

\begin{figure}[h]
    \centering
    \hfill\scalebox{0.8}{\begin{tikzpicture}[
dot/.style={circle, draw=black, fill=black, minimum size=0em, inner sep=0.15em},scale=0.6]

\pgfpointtransformed{\pgfpointxy{1}{1}};
\pgfgetlastxy{\vx}{\vy}

\begin{scope}[node distance=\vy and \vx]
\node[dot] (0) {};
\node[dot] (1) [above left=of 0]{};
\node[dot] (2) [above right=of 0]{};
\node[dot,label=right:$a$] (3) [above=of 2]{};
\node[dot] (4) [above left=of 3]{};
\end{scope}

\draw[-] (0)--(1);
\draw[-] (0)--(2);
\draw[-] (2)--(3);
\draw[-] (3)--(4);
\draw[-] (1)--(4);
\end{tikzpicture}\hspace{3cm}
\begin{tikzpicture}[
dot/.style={circle, draw=black, fill=black, minimum size=0em, inner sep=0.15em},scale=0.8]

\pgfpointtransformed{\pgfpointxy{1}{1}};
\pgfgetlastxy{\vx}{\vy}

\begin{scope}[node distance=\vy and \vx]
\node[dot] (0) {};
\node[dot] (1) [above=of 0]{};
\node[dot] (2) [above=of 1]{};
\node[dot] (3) [left=of 1]{};
\node[dot,label=right:$a$] (4) [right=of 1]{};
\end{scope}

\draw[-] (0)--(1);
\draw[-] (1)--(2);
\draw[-] (3)--(2);
\draw[-] (4)--(2);
\draw[-] (0)--(3);
\draw[-] (0)--(4);
\end{tikzpicture}}\hfill
    \caption{Lattices showing the distinction between the properties \ref{cond:E1} and \ref{cond:E2}}
    \label{fig:e1_e2_different}
\end{figure}
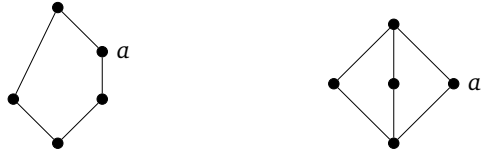

\begin{example} When we take the product of two complete lattices $L$ and $M$, we can embed every $l\in L$ as $(l,0)$, and every $m\in M$ as $(0,m)$. We will then have for each of \hyperref[cond:E1]{(E1-4)} that $(l,0)$ satisfies (E$i$) in $L\times M$ if and only if $l$ satisfies (E$i$) in $L$, and the same for $(0,m)$ in $L\times M$ and $m$ in $M$. This means that the product of the lattice of closed sets of a non-discrete Hausdorff space (Example~\ref{exaE}) and of $M_3$ or $N_5$ (Example~\ref{exaF}) gives a lattice in which \hyperref[cond:E1]{(E1-4)} are all different.
\end{example}

\begin{theorem}\label{thmO}
Let $\mathcal{C}$ be the set of absolutely connected elements in a complete lattice $L$. Then $(L,\mathcal{C})$ is a typical connectivity lattice where Galois closed elements of $\mathcal{D}(\mathcal{C})$ are precisely those totally mail-disconnected sets in $\mathcal{C}$ which are also totally mail-disconnected in $L^+$. Consequently, whenever $(L,\mathcal{C})$ is well-founded, it is separated. 
\end{theorem}

\begin{proof}
 Consider the join $a$ of a non-empty set $X$ of absolutely connected elements that has a lower bound $b$ different from $0$. Suppose $a\leqslant \bigvee S$, where $S$ is as in \ref{cond:E4}. Then $x\leqslant \bigvee S$ for each $x\in X$. Since each $x$ is a absolutely connected, it must be below some $s_x\in S$. Then $b$ is a lower bound for all of $s_x$'s and since $S$ is totally mail-disconnected in $L^+$, this implies that all $s_x$'s are the same element $s$ of $S$. This would force $a$ to be below that very element $s$, thus proving that $a$ is absolutely connected. We have thus established that $(L,\mathcal{C})$ is a typical connectivity lattice. We know already from Theorem~\ref{thmJ} that each Galois closed element of $\mathcal{D}(\mathcal{C})$ is totally mail-disconnected in $L^+$. For the converse, consider an element $S$ of $\mathcal{D}(\mathcal{C})$ that is totally mail-disconnected in $L^+$. Pick any $x\in \mathcal{C}\left(\bigvee S\right)$. Since $x$ is below $\bigvee S$ and is absolutely connected, we must have $x\leqslant s$ for some $s\in S$. This shows that $\mathcal{C}\left(\bigvee S\right)\leqslant S$ and hence $\mathcal{C}\left(\bigvee S\right)=S$. When $(L,\mathcal{C})$ is well-founded, every element of $\mathcal{D}(\mathcal{C})$ is totally mail-disconnected in $L^+$ and hence $(L,\mathcal{C})$ is separated.
\end{proof}

By Theorem~\ref{thmY}, absolutely connected elements in a frame are the same as the connected elements of the frame, in the usual sense. Thanks to the theorem above, some of the results in \cite{BABOOLAL19913} about connected elements in frames can then be deduced from the general theory of connectivity presented here.

In \cite{Serra1998Nov} it is noted that once a subset $\mathcal{C}$ of a lattice $L$ satisfies \ref{cond:CL2}, all atoms and all completely join-prime elements (which in \cite{Serra1998Nov} are referred to as strongly co-prime elements) belong to $\mathcal{C}$ (see, in particular, Proposition 3 and the paragraph after its proof in \cite{Serra1998Nov}). As we establish in the following theorem, the same can be said about absolutely connected elements (which include all completely join-prime elements), once the pair $(L,\mathcal{C})$ is a connectivity lattice. 

\begin{theorem}\label{thmP}
For any Serra connectivity lattice $(L,\mathcal{C})$, the set $\mathcal{C}$ contains all absolutely connected elements. Furthermore, $\mathcal{C}$ is an absolute connectivity in a complete lattice $L$ if and only if $\mathcal{C}$ is the set of all absolutely connected elements in $L$ and every element of $L$ is a join of absolutely connected elements.
\end{theorem}

\begin{proof}
Let $(L,\mathcal{C})$ be a Serra connectivity lattice. Consider an absolutely connected $a\in L$. By Theorem~\ref{thmJ}, $a=\bigvee \mathcal{C}(a)$ where $\mathcal{C}(a)$ is totally mail-disconnected set in $L^+$. Then $a\leqslant s$ for $s\in \mathcal{C}(a)$ and hence $a\in \mathcal{C}(a)$, showing that $a\in\mathcal{C}$. Suppose $(L,\mathcal{C})$ is a separated Serra connectivity lattice. We already know that $\mathcal{C}$ contains all absolutely connected elements. Consider an element $a\in\mathcal{C}$. Suppose $a\leqslant \bigvee S$ where $S$ is totally mail-disconnected in $L^+$. Let $T=\bigcup_{s\in S}\mathcal{C}(s)$. Note that $T$ is totally mail-disconnected in $\mathcal{C}$. Since for each $s\in S$ we have $s=\bigvee\mathcal{C}(s)$, we have that $a\leqslant \bigvee S=\bigvee T$. By separatedness, $T=\mathcal{C}(\bigvee T)$. This implies that $a\leqslant t$ for some $t\in T$, and $t\leqslant s$ for some $s\in S$ where $t\in\mathcal{C}(s)$, thus showing that $a$ is absolutely connected. So $\mathcal{C}$ is the set of all absolutely connected elements. Then also every element $x\in L$ is a join $x=\bigvee \mathcal{C}(x)$ of absolutely connected elements. Finally, suppose $\mathcal{C}$ is the set of all absolutely connected elements. Then we already know by Theorem~\ref{thmO} that $(L,\mathcal{C})$ is a typical connectivity lattice. If every element of $\mathcal{C}$ is a join of absolutely connected elements, then $(L,\mathcal{C})$ will be a Serra connectivity lattice by Theorem~\ref{thmJ}. Furthermore, in this case $(L,\mathcal{C})$ is well-founded and hence separated by Theorem~\ref{thmO}.
\end{proof}

\begin{example}\label{exaP}
    A \emph{locally connected frame} is a frame where every element is a join of elements satisfying \ref{cond:E2}, as per Definition 1.1 in \cite{BABOOLAL19913}. Since conditions \hyperref[cond:E1]{(E1-4)} are equivalent in a frame (Theorem~\ref{thmY}), a frame is locally connected if and only if it is an absolute connectivity lattice. Recall that the frame of open sets of a topological space is locally connected if and only if the topological space is  locally connected (i.e., the set of open connected subsets is a base for the topology, see~\cite{kelly_1955}).
\end{example}

\begin{theorem}
In an absolute connectivity lattice, any element $a$ satisfying \ref{cond:E2} is absolutely connected and hence the conditions \hyperref[cond:E1]{(E1-4)} are equivalent for every element $a$. 
\end{theorem}

\begin{proof}
Consider an element $a$ satisfying \ref{cond:E2}. We have $a=\bigvee \mathcal{C}(a)$, where $\mathcal{C}$ is the set of absolutely connected elements. Consider any element $c\in \mathcal{C}(a)$ and the set $S=\mathcal{C}(a)\setminus\{c\}$. If $c\wedge \bigvee S\neq 0$ then there is an absolutely connected $c'$ below $c\wedge \bigvee S$. Since $S$ is totally mail-disconnected in $L^+$, $c'\leqslant s$ for some $s\in S$ such that $s\neq c$. Then $c'\leqslant c\wedge s=0$, a contradiction. So $c\wedge \bigvee S=0$. But $a=c\vee \bigvee S$, so by \ref{cond:E2}, $a=c$ or $a=\bigvee S$. However, $a=\bigvee S$ is not possible since it forces $c$ to be below one of the elements of $S$. So $a=c$ and $a$ is absolutely connected.
\end{proof}

\begin{example}\label{exaR}
In Example~\ref{exaE}, although every element is a join of elements satisfying \ref{cond:E3}, the lattice is an absolute connectivity lattice only when the topological space is discrete, in which case the lattice is just the power set of the underlying set.
\end{example}

\begin{theorem}\label{thmZ} $\mathcal{D}(\Gamma)$ is an absolute connectivity lattice for any chainmail $\Gamma$. The chainmail $\mathcal{C}$ of absolutely connected elements in $\mathcal{D}(\Gamma)$ is given by the set $\{\{x\}\mid x\in\Gamma\}$ and hence the chainmails $\mathcal{C}$ and $\Gamma$ are isomorphic posets.
\end{theorem}

\begin{proof}
Let $L=\mathcal{D}(\Gamma)$, where $\Gamma$ is a chainmail. We already know that $L$ is a complete lattice (Theorem~\ref{corB}). It is easy to see that a set $S\subseteq L$ is totally mail-disconnected in $L^+$ if and only if the following three conditions hold:
\begin{itemize}
    \item[(i)] $\varnothing\notin S$.
    \item[(ii)] $A\cap B=\varnothing$ for any two distinct elements $A,B\in S$.
    \item[(iii)] $\bigcup S$ is a totally mail-disconnected set in $\Gamma$.
\end{itemize}
Then the join $\bigvee S$ of such $S$ in $L$ is given by $\bigvee S=\bigcup S$. After this it is easy to see that absolutely connected elements of $L$ are precisely the singletons $\{x\}$ where $x\in\Gamma$. Now, consider an arbitrary element $X\in \mathcal{D}(\Gamma)$. Then $S=\{\{x\}\mid x\in X\}$ is totally mail-disconnected in $L^+$ and hence $X=\bigcup S=\bigvee S$, proving that every element of $\mathcal{D}(\Gamma)$ is a join of absolutely connected elements. Theorem~\ref{thmP} then completes the proof.
\end{proof}

The theorem above gives the following fundamental result as a straightforward consequence.

\begin{corollary}\label{corA} There is a one-to-one correspondence between isomorphism classes of absolute connectivity lattices and isomorphism classes of chainmails. In the forward direction, it is given by mapping absolute connectivity lattices to the chainmail of its absolutely connected elements, while in the backward direction it is given by mapping a chainmail to the lattice of its totally mail-disconnected sets.
\end{corollary}

We then also get the following easy but fundamental consequence: it says that in some sense, absolute connectivity is a universal connectivity.

\begin{corollary}
A subset $\mathcal{C}$ of a complete lattice $L$ is a connectivity in $L$ if and only if there is an absolute connectivity lattice $L'$ and a left Galois adjoint map $L'\to L$ with every absolutely connected element of $L'$ being Galois closed, such that $\mathcal{C}$ is the image of absolutely connected elements under the left Galois adjoint map.
\end{corollary}

Another interpretation of the result above is that chainmails are universal posets of connected objects: a map $\mathcal{D}(\Gamma)\approx L'\to L$ such as required in the theorem above is effectively the same as a well-behaved embedding of a chainmail $\Gamma$ in $L$, where the images of such embeddings describe all connectivities $\mathcal{C}$ in $L$. 

We conclude this section with some additional examples of connectivity lattices.

Consider a pair $(L,\mathcal{C})$ where $L$ is a complete lattice and $\mathcal{C}$ is the set of all elements $a$ of $L$ such that \ref{cond:E1} holds. It is not difficult to see that $(L,\mathcal{C})$ is a typical connectivity lattice (a fact similar to the first part Theorem~\ref{thmO}). The following example is a special case of this.

\begin{example}\label{exaT}
The divisibility lattice $L$ of natural numbers (including the number $0$) is a complete lattice (where $0$ is the largest element). A set is totally mail-disconnected in $L^+$ if it is a set of coprime numbers not containing the number $1$. A number $a$ satisfies \ref{cond:E1} if and only if it is $0$ or a non-zero power of a prime number. Since every number is a product of powers of primes, the divisibility lattice of natural numbers is a Serra connectivity lattice. It is not separated, since prime numbers are not the components of their join (when we take infinitely many of them). 
\end{example}

\begin{example}\label{exaU}
When $\Gamma$ is a totally (mail-)disconnected chainmail, $\mathcal{D}(\Gamma)$ is the powerset lattice of the underlying set of $\Gamma$. So absolute connectivity lattices whose chainmails of absolutely connected elements are totally disconnected chainmails are the same as complete atomic Boolean algebras. The absolutely connected elements are then the same as atoms.
\end{example}

The link between atoms and absolutely connected elements in a complete atomic Boolean algebra partially extends to general lattices. If a complete lattice $L$ is well-founded with respect to absolutely connected elements, then clearly every atom is absolutely connected. On the other hand, if in a lattice $L$ every element is a join of atoms, then it is easy to see that every absolutely connected element will be an atom. This yields the following result.

\begin{theorem}\label{thmX}
For an atomic complete lattice $L$, the following conditions are equivalent:
\begin{itemize}
\item[(i)] $L$ is an absolute connectivity lattice.

\item[(ii)] In $L$, absolutely connected elements are the same as atoms.

\item[(iii)] $L$ is a Boolean algebra.
\end{itemize}
\end{theorem}

All of the natural examples of absolute connectivity lattices we encountered so far are frames (and hence, locally connected frames --- see Example~\ref{exaP}). However, it is not true in general that all absolute connectivity lattices are frames, as the following examples show.

\begin{example}\label{exaAA}
Let $\Gamma$ be the chainmail from Example~\ref{exaA}. Totally mail-disconnected sets in $\Gamma$ are the empty set, the singletons, and the set $\{1,4\},\{2,4\},\{3,4\}$. See Figure~\ref{figB} for the Hasse diagrams for $\mathcal{D}(\Gamma)$ and $\Gamma$.
\begin{figure}
$$\scalebox{0.8}{\begin{tikzpicture}[dot/.style={circle, draw=black, fill=black, minimum size=0em, inner sep=0.15em},
    condot/.style={circle, draw=black, fill=black, minimum size=0em, inner sep=0.15em}]
    
    \node[dot, label=right:$\varnothing$] (0) {};
    \node[condot, label=left:$\{1\}$] (1) [above=of 0]{};
    \node[condot, label=right:$\{4\}$] (2) [right=of 1]{};
    \node[condot, label=left:$\{3\}$] (4) [above=of 1]{};
    \node[condot, label=left:$\{2\}$] (3) [left=of 4]{};
    \node[dot, label=right:{{$\{1,4\}$}}] (5) [above=of 2]{};
    \node[dot, label=left:{{$\{2,4\}$}}] (6) [above=of 4]{};
    \node[dot, label=right:{{$\{3,4\}$}}] (7) [above=of 5]{};
    \node[condot, label=left:{{$\{5\}$}}] (8) [above=of 6]{};
    \node[condot, label=right:{{$\{6\}$}}] (9) [above=of 7]{};
    \node[condot, label=right:{{$\{7\}$}}] (10) [above=of 8]{};
    
    \draw[-] (0)--(1)--(3)--(6)--(8)--(10);
    \draw[-] (0)--(2)--(5)--(7)--(9)--(10);
    \draw[-] (1)--(4)--(7)--(8);
    \draw[-] (1)--(5)--(6);
    \end{tikzpicture}\hspace{2cm}
    \begin{tikzpicture}[
    dot/.style={circle, draw=black, fill=black, minimum size=0em, inner sep=0.15em}]
    
    \node[dot,label=right:1] (1) {};
    \node[dot,label=right:3] (3) [above=of 1]{};
    \node[dot,label=right:2] (2) [left=of 3]{};
    \node[dot,label=right:4] (4) [right=of 3]{};
    \node[dot,label=right:5] (5) [above=of 3]{};
    \node[dot,label=right:6] (6) [above=of 4]{};
    \node[dot,label=right:7] (7) [above=of 5]{};
    
    \draw[-] (1)--(2)--(5)--(7);
    \draw[-] (1)--(3)--(5);
    \draw[-] (3)--(6)--(7);
    \draw[-] (4)--(5);
    \draw[-] (4)--(6);
    \end{tikzpicture}}$$
    \caption{Hasse diagrams of an absolute connectivity lattice $\mathcal{D}(\Gamma)$ that is not a frame (left), and the chainmail $\Gamma$ that generates it (right).}
    \label{figB}
\end{figure}
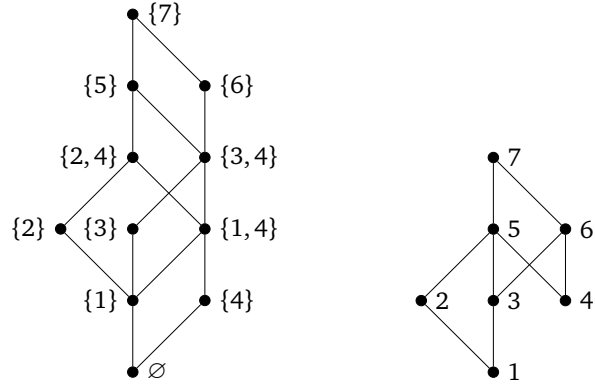
$\mathcal{D}(\Gamma)$ is an absolute connectivity lattice by Theorem~\ref{thmZ}. It is, however, not distributive and hence not a frame. Indeed, 
$$\{4\}\wedge(\{2\}\vee\{3\})=\{4\}\wedge\{5\}=\{4\},$$
while
$$(\{4\}\wedge\{2\})\vee (\{4\}\wedge\{3\})=\varnothing\vee \varnothing=\varnothing.$$
\end{example}

\begin{figure}
    \begin{tikzcd}[row sep={2.1cm,between origins}, column sep={2.1cm,between origins}, cells={nodes={draw, rounded corners, inner sep=3pt}}] & & & & & & &\\
& & & & \txt{preconnectivity\\ \ref{exaN}, \ref{exaO}}\arrow[d, no head] & & & \\
 & & & & \txt{connectivity}\arrow[ddll, no head]\arrow[ddrr, no head] & & & \\
 & & & & & & &\\
 & & \txt{well-founded\\ connectivity}\arrow[ddll, no head]\arrow[dd, no head]\arrow[ddrr, no head] & & & & \txt{not kernel\\ connectivity\\ \ref{exaV}, \ref{exaJ}, \ref{exaM}}\arrow[dl, no head]\arrow[dr, no head] & \\
 & & & & & \txt{typical\\ connectivity\\ \ref{exaH}}\arrow[dl, no head]\arrow[dr, no head] & & \txt{separated\\ connectivity\\ \ref{exaW}}\arrow[dl, no head] \\
 \txt{kernel\\ connectivity\\ \ref{exaI}}\arrow[dr, no head] & & \txt{saturated\\ connectivity}\arrow[dr, no head]\arrow[dl, no head] & & \txt{well-founded\\ typical\\ connectivity\\ \ref{exaY}}\arrow[dl, no head]\arrow[dr, no head] & &\txt{separated\\ typical\\ connectivity\\ \ref{exaG}}\arrow[dl, no head]  & \\
& \txt{degenerate\\ connectivity\\ $\mathcal{C}=L$} \arrow[ddrr, no head]& & \txt{Serra\\ connectivity\\ \ref{exaB}, \ref{exaC}, \ref{exaD}, \ref{exaT}}\arrow[dr, no head] & & \txt{well-founded\\ separated\\ connectivity\\ \ref{exaZ}}\arrow[dl, no head] & & \\
 & & & & \txt{absolute\\ connectivity\\ \ref{exaL}, \ref{exaK}, \ref{exaP}, \ref{exaR},\\ \ref{exaU}, \ref{exaAA}, \ref{exaAB}} & & & \\
 & & & \varnothing\arrow[ur, no head]\arrow[uull, no head] & & & &
\end{tikzcd}
\caption{Taxonomy of connectivity. Numbers are references to examples considered in the paper which  fall in a given taxonomic class but do not fall in the lower classes. These examples include the following connectivities: graph connectivity (Example~\ref{exaB}), hypergraph connectivity (Example~\ref{exaH}), topological connectivity (Example~\ref{exaC}), path-connectedness (Example~\ref{exaD}), trees (Examples~\ref{exaG} and \ref{exaK}), rectangles (Example~\ref{exaO}), incomparable elements (Example~\ref{exaV}), open sets (Example~\ref{exaI}), $k$-connectivity in graphs (Example~\ref{exaJ}), connectivity in locally connected frames, including connected open sets among all open sets (Example~\ref{exaP}), powers of primes (Example~\ref{exaT}), atoms in complete atomic Boolean algebras (Example~\ref{exaU}).}
\label{figA}
\end{figure}
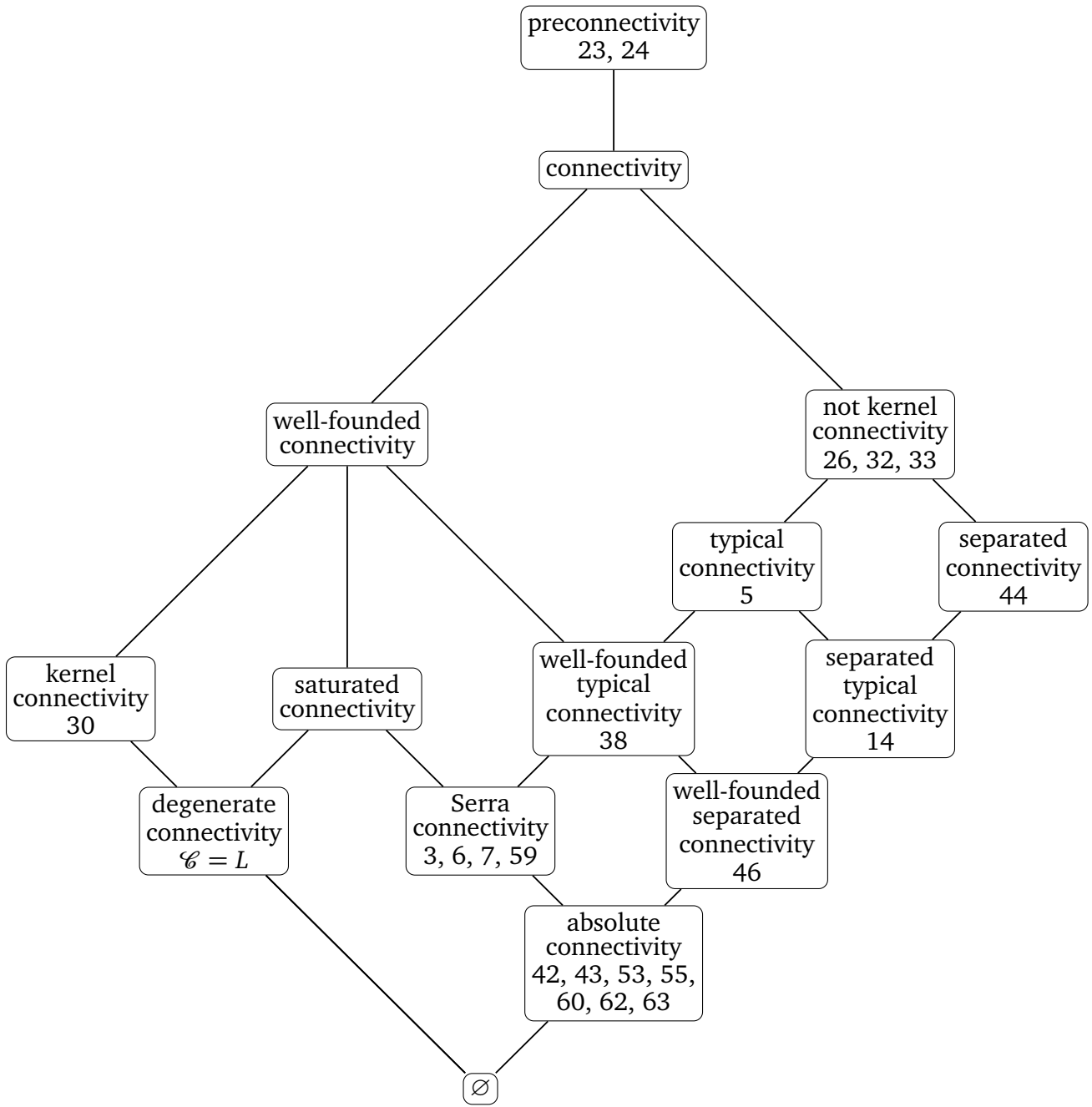

\begin{example}\label{exaAB} Consider any complete lattice $L^+$ and attach to it a new bottom element. The old elements in the new lattice $L$ will be the absolutely connected elements, since the only totally mail-disconnected subsets of $L^+$ are singletons and the empty set. Therefore, $(L,L^+)$ will be an absolute connectivity lattice: $L$ is in fact isomorphic to $\mathcal{D}(L^+)$. When $L^+$ is not a frame, neither will $L$ be a frame.
\end{example}

\section{Concluding Remarks and Some Future Directions}

\textbf{Taxonomy of connectivity.} In Figure~\ref{figA}, we classify each of the examples of preconnectivity lattices explored in this paper in terms of the various properties considered. Note that the poset of the taxonomic groups shown in Figure~\ref{figA} is a complete lattice. Moreover, the results of the paper show that: 
\begin{itemize}
\item meets in this complete lattice match with intersections of the taxonomic groups;

\item hence, the set of these taxonomic groups is a connectivity in the dual of the lattice of all classes of preconnectivities; 

\item every taxonomic group that does come with a reference to an example is the union of the taxonomic groups below it.
\end{itemize}

It would be interesting to enrich Figure~\ref{figA} with more examples, as well as explore other taxonomic groups of connectivity lattices. In particular, extending to lattices various classes of connectivity spaces studied in the literature (see, e.g., \cite{stadler_stadler_2015} and the references there) would be interesting.

It is noteworthy that most of the taxonomic groups in Figure~\ref{figA} can be characterised using natural properties of the connectivity adjunction. Indeed, the following are such characterizations. 
\begin{itemize}
\item Kernel connectivity: the right adjoint does not preserve the bottom element.
\item Not kernel connectivity: the right adjoint preserves the bottom element.
\item Well-founded connectivity: the right adjoint reflects the bottom element.
\item Saturated connectivity: the right adjoint is a right inverse of the left adjoint.
\item Well-founded typical connectivity: the right adjoint preserves and reflects the bottom element.
\item Serra connectivity: the right adjoint is a right inverse of the left adjoint and preserves the bottom element.
\item Separated connectivity: the right adjoint is a left inverse of the left adjoint.
\item Well-founded separated connectivity: the right adjoint is a left inverse of the left adjoint and reflects the bottom element.
\item Absolute connectivity: connectivity adjunction is an isomorphism.
\item Degenerate connectivity: the right adjoint is a right inverse of the left adjoint and does not preserve the bottom element.
\end{itemize}
We do not know a similar characterisation of typical, as well as typical separated connectivity.

\textbf{The category of connectivity lattices.} Note that at the bottom of the taxonomic hierarchy shown in Figure~\ref{figA} are the degenerate connectivity lattices and the absolute connectivity lattices. These are mutually disjoint classes of connectivity lattices, where in each case, the set of connected elements is unique for each complete lattice: in the first case, it is the set of all elements of the lattice, while in the second case it is the set of absolutely connected elements. It turns out that these two canonical types of connectivity lattices provide adjoint to the projections $\pi_1\colon (L,\mathcal{C})\mapsto L$ and $\pi_2\colon (L,\mathcal{C})\mapsto \mathcal{C}$. In detail, we have the following diagram of categories and functors
$$\xymatrix{\mathbf{JLat}\ar@<-2pt>[r]_-{\Delta} & \mathbf{Con}\ar@<-2pt>[l]_-{\pi_1}\ar@<-2pt>[r]_-{\pi_2} & \mathbf{Chm}\ar@<-2pt>[l]_-{\mathcal{D}} }$$
where: 
\begin{itemize}
\item $\pi_1$ is the left adjoint of $\Delta$ as well as its right inverse, 

\item while $\mathcal{D}$ is a left adjoint of $\pi_2$ as well as its left inverse (up to isomorphism). 
\end{itemize}
These categories and functors are defined as follows:  
\begin{itemize}
\item $\mathbf{JLat}$ is a category whose objects are complete lattices and morphisms are arbitrary join-preserving maps between complete lattices. Composition of morphisms is just composition of maps.

\item $\mathbf{Chm}$ is a category whose objects are chainmails and morphisms are maps between chainmails that preserve joins of mails. Composition of morphisms is, again, composition of maps.

\item $\mathbf{Con}$ is a category whose objects are connectivity lattices. A morphism $(L,\mathcal{C})\to (L',\mathcal{C}')$ between connectivity lattices $(L,\mathcal{C})$ and $(L',\mathcal{C}')$ is a morphism $f\colon L\to L'$ in $\mathbf{JLat}$ which preserves connected elements. Composition is defined as in $\mathbf{Lat}$. 

\item $\pi_1(L,\mathcal{C})=L, \;\pi_1(f)=f$.

\item $\Delta(L)=(L,L),\; \Delta(f)=f$. 

\item $\pi_2(L,\mathcal{C})=\mathcal{C}$.

\item On objects, $\mathcal{D}$ is the familiar construction of exterior, except that we want to view $\mathcal{D}(\Gamma)$ as an absolute connectivity lattice, i.e., as the pair $\mathcal{D}(\Gamma)=(\mathcal{D}(\Gamma),\{\{x\}\mid x\in\Gamma\})$. Thinking of elements of $\mathcal{D}(\Gamma)$ as down-closed subchainmails of $\Gamma$, it is clear how we can define a morphism $\mathcal{D}(f)\colon\mathcal{D}(\Gamma)\to \mathcal{D}(\Gamma')$ of connectivity lattices using a morphism $f\colon \Gamma\to\Gamma'$ of chainmails. Namely, for a down-closed subchainmail $\Sigma$ of $\Gamma$, consider the image $f(\Sigma)$ under $f$ and then intersect all down-closed subchainmails of $\Gamma'$ containing $f(\Sigma)$. 
\end{itemize}
These adjunctions are insightful. For instance, the one-to-one correspondence of Corollary~\ref{corA} arises from the category equivalence induced by the adjunction $\mathcal{D}\vdash \pi_2$. Analyzing these adjunctions would be interesting. Furthermore, studying the category $\mathbf{Con}$ and its various subcategories given by taxonomic groups of connectivity lattices would be interesting as well.

\textbf{Taxonomy of chainmails.} It is interesting to characterize chainmails that are isomorphic to chainmails of connected elements of some particular connectivity lattices. For example, chainmails of atoms in complete atomic Boolean algebras are precisely the totally disconnected chainmails (see Theorem~\ref{thmX} and Example~\ref{exaU}). In \cite{upcoming}, we characterize chainmails of connected elements in frames (see Example~\ref{exaP}). Such questions can be difficult, in general. For instance, we have not been able to find a satisfactory characterization of chainmails of connected sets in topological spaces (Example~\ref{exaC}).

\textbf{Connectivity beyond complete lattices?} A connectivity lattice is a pair $(L,\mathcal{C})$ where $L$ is a complete lattice and $\mathcal{C}$ is its subchainmail. This concept may seem more natural if we replace a complete lattice $L$ with a chainmail $\Gamma$ (i.e., we do not force $\Gamma$ to have a bottom element, which is what it takes for it to be a complete lattice). Let us call a pair $(\Gamma,\mathcal{C})$ where $\Gamma$ is a chainmail and $\mathcal{C}$ its subchainmail a  \emph{connectivity poset}. It may appear that a connectivity poset is a more general concept than a connectivity lattice. However, it is not. Every connectivity poset can be represented as a connectivity lattice by considering the pair $$(\mathcal{D}(\Gamma),\mathcal{C}')$$
where $\mathcal{C}'=\{\{x\}\mid x\in\mathcal{C}\}$. The information that $\mathcal{C}$ is a subchainmail of $\Gamma$ can be tracked by requiring that every element of $\mathcal{C}'$ is absolutely connected in $\mathcal{D}(\Gamma)$.

\textbf{Higher connectivity.}
The scope of the theory of connectivity presented in this paper is limited by consideration of `lower connectivity', in the sense that, for instance, our theory does not include consideration of connectivity in categories, where objects are related to each other not by a partial order, but a more complex structure of morphisms and their composition. A theory of `higher connectivity' that generalizes the theory presented in this paper, would, in the context of single-dimensional categories, come close to the work \cite{borger1989} of R.~B\"orger. In what follows we describe the `lower' version of main concepts and results from \cite{borger1989}, by specializing and adapting them to the case when categories there are posets.

A \emph{sink} in a poset $P$ is a pair $(x,B)$ where $x\in P$ and $B\subseteq x^\downarrow$. An element $c\in P$ is \emph{orthogonal} to a sink $(x,B)$ when $c\leqslant x$ if and only if $c\leqslant b$ for a unique $b\in B$. It is easy to see that a  set $\mathcal{C}$ of elements in a complete lattice $L$ is a connectivity in our sense if and only if every element of $\mathcal{C}$ is orthogonal to every pair $(x,\mathcal{C}(x))$. 

A set $\mathcal{C}$ of elements in a poset $P$ is said to be \emph{multicoreflective} in $P$ when for every $x\in P$ there exists a sink $(x,B)$, where $B\subseteq \mathcal{C}$, such that every element of $\mathcal{C}$ is orthogonal to it. It is almost straightforward that when $P$ is a complete lattice, $\mathcal{C}$ is multicoreflective if and only if it is a connectivity.

Consider a set $X$ of elements in a poset $P$. A \emph{local join} of $X$ is an element $x\in P$ such that $X\subseteq x^\downarrow$ and $x$ is a join of $X$ in $y^\downarrow$, for every $y\geqslant x$. As soon as $P$ has a top element, a local join becomes the same as a join. The lower version of Theorem~1.2 in \cite{borger1989} states that for a subset $\mathcal{C}$ of a poset $P$, each of the statements (i-iii) below imply the next and moreover, when every connected subset $X$ of $P$ having an upper bound has a local join, these statements are equivalent to each other.
\begin{itemize}
\item[(i)] $\mathcal{C}$ is multicoreflective in $P$.

\item[(ii)] If an element $a$ of $P$ is orthogonal to every sink $(x,B)$ which is orthogonal to every element of $\mathcal{A}$, then $a\in \mathcal{A}$. 

\item[(iii)] $P$ is closed under local joins of connected subsets.
\end{itemize}
The rest of \cite{borger1989} is concerned with examining various examples of multicoreflectiveness, which relates to earlier work on connectivity in more restrictive categorical contexts (see the references there). It would be interesting to extend the rest of the theory developed in this paper to higher connectivity.

\section*{Acknowledgments}

This paper grew out of the project originally conceptualized by the first author, who invited the third author and Nicholas Sander to join the project, during their undergraduate studies at Stellenbosch University. They invited the second author to serve as the supervisor of their project within the framework of the Foundations of Abstract Mathematics course. The project has developed considerably since its original form. A discussion over dinner between Cerene Rathilal and the second author, during her visit in Stellenbosch, in the framework of the Mathematical Structures and Modelling research programme of the National Institute for Theoretical and Computational Sciences (NITheCS), had an impact on this development. The third author made further progress in his Honors Project in the Division of Mathematics of the Department of Mathematical Sciences at Stellenbosch University, which was supervised by the second author. The paper went through substantial further development to reach its current form. 

This work made use of resources provided by subMIT at MIT Physics and the University of Stellenbosch's \href{http://www.sun.ac.za/hpc}{HPC1 (Rhasatsha)}.

The research of the first author (JFDP) was supported in part by the U.S. Department of Energy, Office
of Science, Office of Nuclear Physics under grant Contract Number DE-SC0011090. The first author would also like to thank MIT for financial support.

The third author would like to thank Stellenbosch University, the Eendrag Merit Bursary, NITheCS, the Skye Foundation, and the Cambridge Trust for financial support.

\bibliographystyle{plain}
\bibliography{refs}

\begin{thebibliography}{10}

\bibitem{BABOOLAL19913}
D.~Baboolal and B.~Banaschewski.
\newblock Compactification and local connectedness of frames.
\newblock {\em Journal of Pure and Applied Algebra}, 70(1):3--16, 1991.

\bibitem{Balachandran1955}
V.~K. Balachandran.
\newblock On complete lattices and a problem of birkhoff and frink.
\newblock {\em Proceedings of the American Mathematical Society}, 6(4):548--553, 1955.

\bibitem{borger_1983}
R.~B{\"o}rger.
\newblock Connectivity spaces and component categories.
\newblock In {\em Categorical topology: Proceedings of the International Conference held at the University of Toledo, Ohio, USA}, volume~5 of {\em Sigma Ser. Pure Math.}, pages 71--89. Heldermann, 1983.

\bibitem{borger1989}
R.~B\"orger.
\newblock Multicoreflective subcategories and coprime objects.
\newblock {\em Topology and its Applications}, 33:127--142, 1989.

\bibitem{braga2003connectivity}
U.~Braga-Neto and J.~Goutsias.
\newblock A theoretical tour of connectivity in image processing and analysis.
\newblock {\em Journal of Mathematical Imaging and Vision}, 19:5--31, 2003.

\bibitem{clementino_janelidze_2024}
M.~Manuel Clementino and G.~Janelidze.
\newblock Effective descent morphisms of filtered preorders.
\newblock {\em Order}, 2024.

\bibitem{upcoming}
J.~F. Du~Plessis, Z.~Janelidze, C.~Rathilal, and B.~A. Wessels.
\newblock On frames and chainmails (in preperation).

\bibitem{dugowson2010connectivity}
S.~Dugowson.
\newblock On connectivity spaces.
\newblock {\em Cahiers de Topologie et G\'eom\'etrie Diff\'erentielle Cat\'egoriques}, LI(4):282--315, 2010.

\bibitem{freyd_scedrov_1990}
P.~Freyd and A.~Scedrov.
\newblock {\em Categories, Allegories}, volume~39 of {\em Mathematical Library}.
\newblock North-Holland, 1990.

\bibitem{Hammer1968Sep}
P.~C. Hammer and W.~E. Singletary.
\newblock {Connectedness-equivalent spaces on the line}.
\newblock {\em Rend. Circ. Mat. Palermo}, 17(3):343--355, September 1968.

\bibitem{kelly_1955}
J.~L. Kelly.
\newblock {\em Topology}.
\newblock The University Series in Higher Mathematics. D.~van~Nostrand~Company, Inc., 1955.

\bibitem{matheron_1985}
G.~Matheron.
\newblock Remarques sur les fermetures-partitions.
\newblock International report, Centre de G\'eostatistique et de Morphologie Math\'ematique, Ecole des Mines de Paris, 1985.

\bibitem{matheron_serra_1988}
G.~Matheron and J.~Serra.
\newblock Strong filters and connectivity.
\newblock In J.~Serra, editor, {\em Image Analysis and Mathematical Morphology}, volume~2, pages 141--157. Academic Press, London, 1988.

\bibitem{article}
J.~Muscat and D.~Buhagiar.
\newblock Connective spaces.
\newblock {\em Series B: Mathematical Science}, 39:1--13, 01 2006.

\bibitem{picado_pultr_2012}
J.~Picado and A.~Pultr.
\newblock {\em Frames and Locales: Topology without Points}.
\newblock Springer, 2012.

\bibitem{OEIS:A374073}
J.~F.~Du Plessis.
\newblock Sequence a374073, number of connected chainmails with $n$ unlabeled elements.
\newblock Online Encyclopedia of Integer Sequences, 2024.
\newblock Accessed: 2025-01-28.

\bibitem{Serra1998Nov}
J.~Serra.
\newblock {Connectivity on Complete Lattices}.
\newblock {\em J. Math. Imaging Vision}, 9(3):231--251, November 1998.

\bibitem{stadler_stadler_2015}
B.~M.~R. Stadler and P.~F. Stadler.
\newblock Connectivity spaces.
\newblock {\em Mathematics in Computer Science}, 9(4):409–436, 2015.

\bibitem{10.2307/1969257}
A.~D. Wallace.
\newblock Separation spaces.
\newblock {\em Annals of Mathematics}, 42(3):687--697, 1941.

\end{thebibliography}
\end{document}